\documentclass[10pt]{article}

\usepackage[margin=1.5in]{geometry}
\usepackage{amsmath}
\usepackage{amssymb}
\usepackage{amsthm}
\usepackage[colorlinks=true,
citecolor=blue,
linkcolor=red,
urlcolor=black,
anchorcolor=black]{hyperref}
\usepackage{xcolor}
\usepackage[english]{babel}
\usepackage{graphicx}
\usepackage{longtable}

\newtheorem{theorem}{Theorem}[section]
\newtheorem{prop}[theorem]{Proposition}
\newtheorem{lemma}[theorem]{Lemma}
\newtheorem{cor}[theorem]{Corollary}
\newtheorem{defn}[theorem]{Definition}

\theoremstyle{definition}
\newtheorem{remark}[theorem]{Remark}
\newtheorem{example}[theorem]{Example}

\title{Cusps and Commensurability Classes of Hyperbolic 4-Manifolds}
\author{Connor Sell}
\date{ }
\begin{document}
	
	\maketitle
	
	\begin{abstract}
		There are six orientable, compact, flat 3-manifolds that can occur as cusp cross-sections of hyperbolic 4-manifolds.  This paper provides criteria for exactly when a given commensurability class of arithmetic hyperbolic 4-manifolds contains a representative with a given cusp type.  In particular, for three of the six cusp types, we provide infinitely many examples of commensurability classes that contain no manifolds with cusps of the given type; no such examples were previously known for any cusp type.
	\end{abstract}
	\section{Introduction}
	\label{sec:intro}
	Let $ M = \mathbb{H}^n / \Gamma $ be a finite-volume noncompact hyperbolic $ n $-manifold.  A cusp of $ M $ is homeomorphic to $ B \times \mathbb{R}^+ $, where $ B $ is a compact flat $ (n-1) $-manifold.  If $ M $ is orientable, then $ B $ must be orientable.  In \cite{LR}, Long and Reid proved that every compact flat $ (n-1) $-manifold up to homeomorphism must occur as a cusp cross-section of a hyperbolic $ n $-orbifold; this result was upgraded from $ n $-orbifolds to $ n $-manifolds by McReynolds in \cite{McR}.  The paper \cite{LR} gives a constructive algorithm which, given a compact flat $ (n-1) $-manifold, outputs an arithmetic hyperbolic $ n $-orbifold with a cusp with the specified cross-section.  We discuss this algorithm in more detail in \S \ref{sec:cusp_yes}.
	
	For ease of notation, we may refer to a cusp with cross-section $ B $ as a \textit{cusp of type $ B $}, as the cross-section of a cusp determines its homeomorphism class.  We may also refer to a homeomorphism class of cusps, or ``cusp type,'' by its cross-section.  See \S \ref{sec:flat} for a description of the six possible cusp types for hyperbolic 4-manifolds, and the names used below.
	
	The above results tell us that each compact flat $ (n-1) $-manifold occurs as a cusp of some hyperbolic $ n $-manifold, but little is known about under which conditions each cusp type can occur.  To investigate the occurrence of cusp types further, it makes sense to look at compact flat 3-manifolds in finite-volume hyperbolic 4-manifolds, as this is the lowest dimension in which multiple orientable cusp types can occur.  It is well-known that the 3-torus occurs as a cusp in every commensurability class of cusped hyperbolic 4-manifolds.  Indeed, in every commensurability class of cusped hyperbolic 4-manifolds, manifolds with all cusp types being the 3-torus occur \cite{MRS}.  A striking result of Kolpakov and Martelli \cite{KM} proved that there exist one-cusped hyperbolic 4-manifolds have cusp type the 3-torus.  Furthermore, \cite{KS} shows the $ \frac{1}{2} $-twist cusp also occurs as the cusp type of a one-cusped hyperbolic 4-manifold.  On the other hand, the $ \frac{1}{3} $-twist and $ \frac{1}{6} $-twist have been obstructed from occurring as cusps of one-cusped manifolds \cite{Eta}.  Although it is as yet unknown whether the Hantzsche-Wendt manifold occurs as a cusp type of a one-cusped hyperbolic 4-manifold, it was shown in \cite{FKS} that there exists a finite-volume hyperbolic 4-manifold with all cusp types being the Hantsche-Wendt manifold.  We also note that the isometry classes within each homeomorphism class that occur geometrically as cusps of hyperbolic 4-manifolds are dense in the moduli space of any compact flat 3-manifold \cite{Nimer}.
	
	In this paper, we provide the first known examples of commensurability classes that avoid three cusp types.  In fact, we provide infinitely many such examples, obtaining the result below.  Furthermore, given any commensurability class $ C $ of cusped arithmetic hyperbolic 4-manifolds and any cusp type $ B $, we will give conditions on when $ C $ contains a manifold with a cusp of type $ B $ in Theorem \ref{theorem:full_pre}.  Notably, three cusp types occur in every such class.  We refer to \S \ref{sec:qfqa} for terminology used in Theorem \ref{theorem:infinite}.
	
	%This paper also focuses on orientable hyperbolic 4-manifolds, for which there are six possible cusps up to homeomorphism; see \S \ref{sec:arith_hyp_mflds} for a discussion.  For ease of notation, we may refer to a cusp with cross-section $ B $ as a \textit{cusp of type $ B $}, as the cross-section of a cusp determines its homeomorphism class.  We may also refer to a homeomorphism class of cusps, or ``cusp type,'' by its cross-section.  It is well-known that the 3-torus occurs as a cusp in every commensurability class of cusped hyperbolic 4-manifolds (see for example \cite{MRS}), but it has been an open question whether this is true for any other cusp types.  In this paper, we provide the first known examples of commensurability classes that avoid some cusp types, answering this question in the negative for those types.  Furthermore, given any commensurability class $ C $ of arithmetic hyperbolic 4-manifolds and any cusp type $ B $, we give conditions on when $ C $ contains a manifold with a cusp of type $ B $.  We refer to \S \ref{sec:qfqa} for terminology used in Theorem \ref{theorem:full_pre}.
	\begin{theorem}
		\label{theorem:infinite}
		Every commensurability class of arithmetic hyperbolic 4-manifolds contains manifolds with the 3-torus, $ \frac{1}{2} $-twist, and Hantzsche-Wendt manifold cusps.  There exist infinitely many commensurability classes $ C $ of hyperbolic 4-manifolds such that no manifold in $ C $ has a cusp of type $ \frac{1}{3} $-twist (resp. $ \frac{1}{4} $-twist, $ \frac{1}{6} $-twist).
	\end{theorem}
	Additionally, we can use ``inbreeding'' of arithmetic hyperbolic 4-manifolds \cite{Agol} to construct some non-arithmetic manifolds that avoid some cusp types, up to commensurability.
	\begin{theorem}
		\label{theorem:non-arith}
		There exist infinitely many commensurability classes of finite-volume, cusped, non-arithmetic hyperbolic 4-manifolds that avoid each of the following cusp types: the $ \frac{1}{3} $-twist, the $ \frac{1}{4} $-twist, and the $ \frac{1}{6} $-twist.
	\end{theorem}
	We briefly review the organization of the paper.  In \S \ref{sec:qfqa}, \S \ref{sec:arith_hyp_mflds}, and \S \ref{sec:flat}, we provide preliminary information about quadratic forms, quaternion algebras, arithmetic hyperbolic manifolds, and the six orientable compact flat 3-manifolds that are the candidates for cusp types of orientable hyperbolic 4-manifolds.  In \S \ref{sec:cusp_yes} and \S \ref{sec:cusp_no}, we prove Theorem \ref{theorem:infinite} and generalize it to give complete conditions on when a given commensurability class contains a manifold with a cusp of given type.  In \S \ref{sec:5dim}, we use this result to show that there are some commensurability classes of hyperbolic 5-manifolds that avoid some compact flat 4-manifold cusp types, and explain why we can't make the same argument in higher dimensions.  In \S \ref{sec:non-arith}, we show that there are commensurability classes of non-arithmetic hyperbolic manifolds in both 4 and 5 dimensions that avoid certain cusp types as well, proving Theorem \ref{theorem:non-arith}.
	\\ ~ \\
	\noindent \textbf{Acknowledgements}.  The author wishes to thank his PhD advisor Alan Reid for his guidance and useful comments.  This paper is partially supported by NSF grant DMS-1745670.
	\section{Quadratic forms and quaternion algebras}
	\label{sec:qfqa}
	\subsection{Quadratic forms}
	\begin{defn}[Quadratic form]
		A \textit{quadratic form} over a field $ K $ is a homogeneous polynomial of degree 2 with coefficients in $ K $.
	\end{defn}
	A quadratic form $ q(x) = \sum_{i=1}^n \sum_{j=1}^n a_{ij} x_i x_j $ in $ n $ variables is said to have \textit{rank} $ n $, and can be written as an $ n \times n $ symmetric matrix $ Q $ such that $ q(x) = x^tQx $.  This can be accomplished by setting the entries $ Q_{ii} = a_{ii} $ and $ Q_{ij} = \frac{a_{ij}}{2} $ when $ i \neq j $.
	
	For any quadratic form $ q $ of rank $ n $ and ring $ R $ we can define the orthogonal group $ O(q, R) $ to be the group of all invertible $ n \times n $ matrices $ A $ with entries in $ R $ such that $ q(x) = q(Ax) $ for any $ x \in R^n $.  We can similarly define the special orthogonal group $ SO(q, R) $ to be the subgroup of $ O(q,R) $ of matrices with determinant 1.  Note that $ SO(q, \mathbb{R}) $ is a Lie group, and thus has an identity component $ SO_0(q, \mathbb{R}) $.  Then for any subring $ R \subset \mathbb{R} $, we define $ SO_0(q,R) = SO_0(q, \mathbb{R}) \cap SO(q, R) $.  Our focus is quadratic forms over  $ \mathbb{Q} $ and the corresponding groups $ SO_0(q, \mathbb{Z}) $.
	\begin{defn}[Rational equivalence]
		Two quadratic forms given by matrices $ Q_1, Q_2 \in GL(n, \mathbb{Q}) $ are \textit{rationally equivalent} (or \textit{equivalent over $ \mathbb{Q} $}) if there exists $ T \in GL(n, \mathbb{Q}) $ such that $ T^t Q_1 T = Q_2 $.
	\end{defn}
	All quadratic forms over $ \mathbb{Q} $ are rationally equivalent to a \textit{diagonal} quadratic form, that is, a quadratic form whose corresponding matrix is diagonal.  Thus, when working with a rational equivalence class of quadratic forms, we will always choose a diagonal representative.  For ease of notation, we will denote diagonal quadratic forms $ q(x) = \sum_{i=1}^n a_i x_i^2 $ by writing their coefficients $ \langle a_1, \ldots, a_n \rangle $.  All quadratic forms in this paper will be nondegenerate; that is, all $ a_i \neq 0 $.
	
	There is another relevant notion of equivalence closely related to rational equivalence \cite{MA}.
	\begin{defn}[Projective equivalence]
		Two quadratic forms $ q_1 $ and $ q_2 $ are \textit{projectively equivalent over $ \mathbb{Q} $}, or just ``projectively equivalent,'' if there are two nonzero integers $ a $ and $ b $ such that $ aq_1 $ and $ bq_2 $ are rationally equivalent.
	\end{defn}
	Let $ q_1  $ and $ q_2 $ be two quadratic forms with the same signature and discriminant.  If $ q_1 $ and $ q_2 $ have odd rank, then they are projectively equivalent if and only if they are rationally equivalent.  Later, this will allow us to check for projective equivalence by scaling two forms of odd rank so they have the same discriminant, and then checking for rational equivalence.
	
	A complete set of invariants for diagonal quadratic forms up to rational equivalence is given by the signature, discriminant, and the Hasse-Witt invariants over all primes $ p $.  A quadratic form $ q = \langle a_1, \ldots, a_n \rangle $ of \textit{signature} $ (a,b) $ has $ a $ positive coefficients and $ b $ negative coefficients.  The \textit{discriminant} $ d \in \mathbb{Q} / (\mathbb{Q}^\times)^2 $ is given by $ d = \prod_{i=1}^n a_i $; note that it is defined only up to multiplication by squares.  The Hasse-Witt invariants are a little harder to define, and contain the bulk of the number-theoretic information.  For integers $ a $ and $ b $ and prime $ p $, we first define the Hilbert symbol
	\begin{equation*}
		(a,b)_p = \Big\{ \begin{array}{ll} 1 & \text{if~} z^2 = ax^2 + by^2 \text{~has~a~solution~in~} \mathbb{Q}_p \\ -1 & \text{otherwise.} \end{array}
	\end{equation*}
	Notation: $ \mathbb{Q}_p $ will denote the $ p $-adic field at $ p $, or $ \mathbb{R} $ if $ p = \infty $.
	\begin{defn}[Hasse-Witt invariant] Given a diagonal quadratic form $ q = \langle a_1, \ldots, a_n \rangle $ over $ \mathbb{Q} $ and a prime $ p $, possibly $ \infty $, the \textit{Hasse-Witt invariant of $ q $ at $ p $} is given by
		\begin{equation*}
			\epsilon_p(q) = \prod_{1 \leq i < j \leq n} (a_i, a_j)_p.
		\end{equation*}
	\end{defn}
	Every Hasse-Witt invariant must have value $ 1 $ or $ -1 $.  There is a closed-form equation that allows us to easily compute a Hilbert symbol; thus, a Hasse-Witt invariant is easy to compute as well.  Let $ a = p^\alpha u $ and $ b = p^\beta v $ with $ u $ and $ v $ both relatively prime to $ p $ in $ \mathbb{Z} $.  Then for $ p > 2 $, we have
	\begin{equation*}
		(a,b)_p = (-1)^{\alpha \beta \tau(p)} \left( \frac{u}{p} \right)^\beta \left( \frac{v}{p} \right)^\alpha,
	\end{equation*}
	and for $ p = 2 $,
	\begin{equation*}
		(a,b)_p = (-1)^{\tau(u) \tau(v) + \alpha \omega(v) + \beta \omega(u)}.
	\end{equation*}
	Here, we use the Legendre symbol, and the functions $ \tau(x) = \frac{x-1}{2} $ and $ \omega(x) = \frac{x^2 - 1}{8} $, both of which only need to be defined modulo 2. \cite[Ch. III, Thm. 1]{Serre}
	
	We can see from these equations that $ (a,b)_p $ can only be $ -1 $ if either $ a $ or $ b $ is divisible by $ p $ an odd number of times.  This means that $ \epsilon_p(q) = 1 $ for all primes $ p $ that don't occur as a factor of a coefficient of $ q $.  In particular, for any given quadratic form $ q $, $ \epsilon_p(q) = 1 $ for all but finitely many values of $ p $.
	
	Additionally, Hilbert's reciprocity law states that the Hilbert symbols satisfy the identity $ \prod_p (a,b)_p = 1 $, where the product is taken over all places $ p $ of $ \mathbb{Q} $, including $ p = \infty $ \cite[Ch. III, Thm. 3]{Serre}.  From this, we deduce the identity $ \prod_p \epsilon_p(q) = 1 $ for any quadratic form $ q $.  Since $ (a,b)_\infty $ depends on the existence of a nonzero solution to $ z^2 = ax^2 + by^2 $ over the field $ \mathbb{Q}_\infty = \mathbb{R} $, we know $ (a,b)_\infty = -1 $ if and only if both $ a $ and $ b $ are negative.  We'll be working mostly with quadratic forms of signature $ (4,1) $, so in this case $ \prod_{1 \leq i < j \leq n} (a_i,a_j)_\infty = 1 $, as no pair $ (a_i, a_j) $ of distinct coefficients are both negative.  As a result, the identity $ \prod_p \epsilon_p(q) = 1 $ holds when we consider only finite places $ p $ for quadratic forms of signature $ (4,1) $.
	\subsection{Quaternion algebras}
	\begin{defn}[Quaternion algebra]
		A \textit{quaternion algebra} over a field $ F $ with $ \text{char}(F) \neq 2 $ is an algebra consisting of elements $ w + xi + yj + zij $, with $ w, x, y, z \in F $, equipped with relations $ i^2 = a $, $ j^2 = b $, and $ ij = -ji $ for some fixed $ a, b \in F $.  We write this as $ \left( \frac{a,b}{F} \right) $.
	\end{defn}
	Alternatively, a quaternion algebra $ Q $ over $ F $ is any central simple algebra of dimension 4 over $ F $.  Every such $ Q = \left( \frac{a,b}{F} \right) $ has a \textit{norm form}, given by $ N(w + xi + yj + zij) = w^2 - ax^2 - by^2 + abz^2 $, which is compatible with multiplication in $ Q $.
	
	The \textit{pure quaternions} $ Q_0 $ of $ Q $ are the elements $ w + xi + yj + zij $ with $ w = 0 $.  Restricted to the pure quaternions, the norm form of $ Q $ (or for short, the norm form of $ Q_0 $) becomes $ N(xi + yj + zij) = -ax^2 - by^2 + abz^2 $.  Note that any quadratic form of rank 3 and discriminant 1 is projectively equivalent to such a form.  To see this, observe that if $ \langle -a, -b, c \rangle $ has discriminant $ 1 $, then $ c = ab $ up to multiplication by a square.  In particular, the quadratic form $ \langle a, b, ab \rangle $ coincides with the norm form of $ \left( \frac{-a, -b}{\mathbb{Q}} \right) $.  In this paper, we will make use of quadratic forms of signature $ (4,1) $ that are the direct sum of a positive definite norm form of some $ Q_0 $ and $ \langle 1, -1 \rangle $.
	\begin{defn}[Quaternion type]
		A \textit{quadratic form of quaternion type} is a quadratic form $ q = \langle a, b, ab, 1, -1 \rangle $ for some positive $ a, b \in \mathbb{Z} $.
	\end{defn}
	\begin{lemma}
		\label{lemma:all_quat}
		Every quadratic form $ q $ over $ \mathbb{Q} $ of signature $ (4,1) $ is projectively equivalent to a quadratic form $ q^\prime $ of quaternion type.
	\end{lemma}
	In order to prove this lemma, we'll need to use Conway's $ p $-excesses, as described in \cite[Chapter 15]{CS}.  These will not appear in the rest of the paper, so readers not interested in the proof of this lemma may ignore these definitions.
	\begin{defn}[$ p $-excess of rank 1 quadratic form]
		Let $ p \neq 2 $ be a prime, possibly $ \infty $, and let $ q = \langle a \rangle $ be a quadratic form such that $ a = p^k u $ with $ u $ relatively prime to $ p $.  If $ p = \infty $, then let $ p^k $ be the sign of $ a $ and $ u $ its magnitude.  Then we define the $ p $-excess of $ q $ to be
		\begin{equation*}
			e_p(q) \equiv \bigg\{
			\begin{array}{l}
				p^k + 3 \text{~(mod 8) if k is odd and u is a quadratic nonresidue mod p} \\
				p^k - 1 \text{~(mod 8) otherwise.}
			\end{array}
		\end{equation*}
		If $ p = 2 $, then
		\begin{equation*}
			e_p(q) \equiv \bigg\{
			\begin{array}{l}
				-p^k - 3 \text{~(mod 8) if k is odd and u} \equiv 3, 5 \text{~(mod~8)} \\
				-p^k + 1 \text{~(mod 8) otherwise.}
			\end{array}
		\end{equation*}
	\end{defn}
	\begin{defn}[$ p $-excess of arbitrary quadratic form]
		Let $ p $ be a prime, possibly $ \infty $, and let $ q = \langle a_1, \ldots, a_n \rangle $ be a quadratic form.  Then we define the $ p $-excess of $ q $ to be
		\begin{equation*}
			e_p(q) \equiv \sum_{i=1}^n e_p(\langle a_i \rangle) \text{~(mod~8)}.
		\end{equation*}
	\end{defn}
	The most notable properties of the $ p $-excesses are that they are additive under direct sum of quadratic forms, and that they are invariant under rational equivalence.  In fact, $ p $-excesses are part of a complete invariant of quadratic forms up to rational equivalence together with the signature and, in the case of forms of even rank, the discriminant. \cite[\S 15.5.1, Thm. 3]{CS}  We can also extract the Hasse-Witt invariants of a quadratic form $ q $ from the discriminant $ d $ and $ p $-excesses $ e_p(q) $ \cite[\S 15.5.3]{CS}.
	\begin{equation*}
		\epsilon_p(q) = \bigg\{
		\begin{array}{ll}
			1 & \text{~if~} e_p(q) = e_p(\langle d, 1, \ldots, 1 \rangle) \\
			-1 & \text{~otherwise}
		\end{array}
	\end{equation*}
	To prove Lemma \ref{lemma:all_quat}, we'll use the additivity of $ e_p $ to construct a rank 3 form $ q_3 $ of discriminant 1 such that $ q_3 \oplus \langle 1, -1 \rangle $ has certain desired Hasse-Witt invariants.  We'll also use the following Lemma, which can be found in greater generality in \cite[Ch. IV, Prop. 7]{Serre}.
	\begin{lemma}
		\label{lemma:serre}
		Let $ d $, $ r $, $ s $, and $ n $ be integers, and $ \epsilon_p $ be $ 1 $ or $ -1 $ for each prime $ p $, including $ \infty $.  Then there exists a rank $ n $ quadratic form $ q $ of discriminant $ d $, signature $ (r,s) $, and Hasse-Witt invariants $ \epsilon_p $ if and only if the following conditions are satisfied.
		\begin{enumerate}
			\item $ \epsilon_p = 1 $ for almost all $ p $ and $ \prod \epsilon_p = 1 $ over all primes $ p $.
			\item $ \epsilon_p = 1 $ if $ n = 1 $ or if $ n = 2 $ and the image of $ d $ in $ \mathbb{Q}_p^\ast / (\mathbb{Q}_p^\ast)^2 $ is $ -1 $.
			\item $ r, s \geq 0 $ and $ n = r+s $.
			\item The sign of $ d $ is equal to $ (-1)^s $.
			\item $ \epsilon_\infty = (-1)^{\frac{s(s-1)}{2}} $.
		\end{enumerate}
	\end{lemma}
	\begin{proof}[Proof of Lemma \ref{lemma:all_quat}]
		We can scale $ q $ to ensure that it has discriminant $ -1 $ by multiplying the entire form by $ -d $, where $ d $ is its discriminant.  This will multiply the product of the terms by $ -d^5 $, and thus we'll obtain the new discriminant $ -d^6 \equiv -1 $.  Note that scaling a form does not change its projective equivalence class.
		
		Now, compute the $ p $-excesses $ e_p(q) $ and set $ e^\prime_p = e_p(q) - e_p(\langle 1,-1 \rangle) $.  By definition, we know $ e_p(\langle a_1, a_2, \ldots, a_n \rangle) = \sum_{i=1}^n e_p(a_i) $, so if we can find a quadratic form $ q_3 $ of signature $ (3,0) $, discriminant 1, and $ p $-excesses equal to $ e^\prime_p $, then $ q^\prime = q_3 \oplus \langle 1,-1 \rangle $ will have $ p $-excesses equal to those of $ q $ and discriminant $ -1 $.  Since it will also have signature $ (4,1) $,  $ q^\prime $ will be rationally equivalent to $ q $.
		
		It suffices, then, to show that $ q_3 $ exists.  Lemma \ref{lemma:serre} gives five conditions on the Hasse-Witt invariants, signature, and discriminant under which a quadratic form must exist.  Conditions (2)-(4) hold trivially for $ q_3 $, either because they only apply to forms of rank 2 or less, or because they merely require that the signature is valid and agrees with the sign of the discriminant.  This is true because the desired signature of $ q_3 $ is $ (3,0) $ and the discriminant is $ 1 $.
		
		Conditions (1) and (5) concern the desired Hasse-Witt invariants $ \epsilon_p^\prime $ of $ q_3 $, which can be determined from the desired discriminant $ 1 $ and desired $ p $-excesses $ e_p^\prime $.  We will show that $ \epsilon_p^\prime = \epsilon_p(q) $ for all $ p $, and thus that conditions (1) and (5) are satisfied.
		
		Recall that $ \epsilon_p^\prime = 1 $ if $ e_p^\prime = e_p(\langle d(q_3), 1, 1 \rangle) = e_p(\langle 1,1,1 \rangle) $, and $ -1 $ otherwise.  Similarly, we can compute the Hasse-Witt invariants of $ q $ to be $ \epsilon_p(q) = 1 $ if and only if $ e_p(q) = e_p(\langle -1,1,1,1,1 \rangle) $.  By construction, $ e_p^\prime = e_p(q) - e_p(\langle 1,-1 \rangle) $, and note that $ e_p(\langle 1,1,1 \rangle) = e_p(\langle -1,1,1,1,1 \rangle) - e_p(\langle 1,-1 \rangle) $ by additivity of $ p $-excesses.  Thus, $ \epsilon_p^\prime = \epsilon_p(q) $ for all $ p $.
		
		In particular, we know that for any quadratic form $ q $, $ \epsilon_p(q) = 1 $ for all but finitely many $ p $, and $ \prod \epsilon_p(q) = 1 $ over all primes $ p $.  These same properties must now hold for $ \epsilon_p^\prime $, so condition (1) holds.  Similarly, we know $ \epsilon_{\infty}(q) = 1 $ since $ q $ has signature (4,1), so $ \epsilon_{\infty}^\prime = 1 $ as well, satisfying condition (5).  Now we can apply Lemma \ref{lemma:serre} to deduce that a valid quadratic form $ q_3 $ exists with signature $ (3,0) $, discriminant $ 1 $, and Hasse-Witt invariants $ \epsilon_p(q_3) = \epsilon_p^\prime $.
		
		As stated above, we can now take the form $ q^\prime = q_3 \oplus \langle 1,-1 \rangle $, which is rationally equivalent to $ q $, has discriminant $ -1 $, and is of the form $ \langle a,b,c,1,-1 \rangle $, where $ q_3 = \langle a,b,c \rangle $.
	\end{proof}
	On the pure quaternions of any quaternion algebra, we can define the orthogonal group
	\begin{equation*}
		O(N, Q_0) = \{f: Q_0 \rightarrow Q_0 | f \text{~linear~}, N(f(x)) = N(x) \text{~for~all~} x \in Q_0 \}
	\end{equation*}
	as the set of linear transformations on $ Q_0 $ that preserve the norm form.  These transformations can be described as conjugation by the units $ Q^\ast $ of $ Q $.  This is the intuition behind the following theorem from \cite[\S~2.4]{MR}.
	\begin{theorem}
		\label{theorem:quad_to_quat}
		Let $ Q = \left( \frac{-a, -b}{\mathbb{Q}} \right) $ and $ q = \langle a, b, ab \rangle $.  Then $ SO(q, \mathbb{Q}) $ is isomorphic to $ Q^\ast / Z(Q^\ast) $, where $ Z(G) $ denotes the center of $ G $.
	\end{theorem}
	There are three more theorems from \cite{MR} that are used in our argument.  We state them here, along with a relevant definition, and remark that it will be important to obstruct certain torsion from occurring in $ Q^\ast / Z(Q^\ast) $.
	\begin{defn}[Ramification]
		A prime $ p $ ramifies a quaternion algebra $ Q $ over $ \mathbb{Q} $ if $ Q \otimes_{\mathbb{Q}} \mathbb{Q}_p $ is isomorphic to the unique division algebra of dimension 4 over $ \mathbb{Q}_p $.  Otherwise, $ Q \otimes_{\mathbb{Q}} \mathbb{Q}_p $ is isomorphic to the algebra of $ 2 \times 2 $ matrices $ M_2(\mathbb{Q}_p) $ and we say $ Q $ splits over $ p $.
	\end{defn}
	\begin{theorem}[{\cite[Lemma~12.5.6]{MR}}]
		\label{theorem:12.5.6}
		Let $ n > 2 $, $ \xi_n $ be a primitive $ n $th root of unity, and $ Q $ be a quaternion algebra over $ \mathbb{Q} $.  Then $ Q^\ast / Z(Q^\ast) $ contains an element of order $ n $ if and only if $ \xi_n + \xi_n^{-1} \in \mathbb{Q} $ and $ \mathbb{Q}(\xi_n) $ embeds in $ Q $.
	\end{theorem}
	\begin{theorem}[{\cite[Thm.~7.3.3]{MR}}]
		\label{theorem:7.3.3}
		Given a quaternion algebra $ Q $ over $ \mathbb{Q} $ and a quadratic extension $ L $ of $ \mathbb{Q} $, then $ L $ embeds in $ Q $ if and only if for each prime $ p $ that ramifies $ Q $, $ p $ does not split in $ L $.
	\end{theorem}
	\begin{theorem}[{\cite[Thm.~2.6.6]{MR}}]
		\label{theorem:2.6.6}
		Let $ p \neq 2, \infty $ be a prime in $ \mathbb{Q} $.  Consider the quaternion algebra $ Q = \left( \frac{a,b}{\mathbb{Q}} \right) $, with both $ a $ and $ b $ squarefree.
		\begin{enumerate}
			\item If $ p $ does not divide $ a $ or $ b $, then $ p $ does not ramify $ Q $.
			\item If $ p $ divides $ a $ but not $ b $, then $ p $ ramifies $ Q $ if and only if $ b $ is a quadratic residue mod $ p $.
			\item If $ p $ divides both $ a $ and $ b $, then $ p $ ramifies $ Q $ if and only if $ -a^{-1}b $ is a quadratic residue mod $ p $.
		\end{enumerate}
	\end{theorem}
	\section{Arithmetic Hyperbolic manifolds}
	\label{sec:arith_hyp_mflds}
	\subsection{Hyperbolic manifolds}
	\label{subsec:hyp_mflds}
	Let $ q = x_1^2 + \ldots + x_n^2 - x_{n+1}^2 $, a quadratic form of signature $ (n,1) $.  We define hyperbolic space using the hyperboloid model: $ \mathbb{H}^n = \{ x \in \mathbb{R}^{n+1} | q(x) = -1, x_{n+1} > 0 \} $, equipped with the metric derived from the inner product $ x \circ y = \sqrt{x_1 y_1 + \ldots + x_n y_n - x_{n+1} y_{n+1}} $ such that $ (x \circ x)^2 = q(x) $.  A hyperplane in $ \mathbb{H}^n $ is an intersection of a subspace $ V \subset \mathbb{R}^{n+1} $ with $ \mathbb{H}^n $, and $ \mathbb{H}^n $ has a boundary $ \partial \mathbb{H}^n $ consisting of 1-dimensional subspaces of \textit{light-like vectors} $ y \in \mathbb{R}^{n+1} $ such that $ q(y) = 0 $.  The isometries of $ \mathbb{H}^n $ must preserve $ q $, and in fact $ \text{Isom}^+(\mathbb{H}^n) = SO_0(q, \mathbb{R}) $.
	
	Observe that we can perform this construction with any arbitrary form $ q^\prime $ of signature $ (n,1) $ in the place of $ q $.  The resulting space $ \mathbb{H}^n_{q^\prime} $ is isometric to $ \mathbb{H}^n $, although both are different subsets of $ \mathbb{R}^{n+1} $ and points in $ \mathbb{Q}^{n+1} $ in one model may not correspond to points in $ \mathbb{Q}^{n+1} $ in the other.  Thus, $ \text{Isom}^+(\mathbb{H}^n) $ is isomorphic to $ \text{Isom}^+(\mathbb{H}^n_{q^\prime}) $.  In particular, there is a linear transformation $ A $ that maps any $ \mathbb{H}^n_{q^\prime} $ to $ \mathbb{H}^n $ isometrically, so any isometry $ \gamma \in \text{Isom}^+(\mathbb{H}^n_{q^\prime}) $ can be said to sit in $ \text{Isom}^+(\mathbb{H}^n) $ as $ A \gamma A^{-1} $.  We will sometimes abuse notation and refer to any $ \mathbb{H}^n_{q^\prime} $ as $ \mathbb{H}^n $, when it is clear which quadratic form is being used.
	
	We will use the notion of hyperplanes $ P $ sitting \textit{rationally} inside $ \mathbb{H}^n_q $.  By this, we mean $ P $ is the intersection of $ \mathbb{H}^n_q $ with a subspace $ V \subset \mathbb{R}^{n+1} $ determined by a system of equations with rational coefficients.  This notion depends on our choice of $ q $, which is our case will always have coefficients in $ \mathbb{Z} $.
	
	A hyperbolic $ n $-manifold is a quotient $ \mathbb{H}^n / \Gamma $ of hyperbolic $ n $-space by a discrete, torsion-free group $ \Gamma $ acting on $ \mathbb{H}^n $ via isometries.  If $ \Gamma $ is not torsion-free, a hyperbolic orbifold results instead.  A cusp of a finite-volume hyperbolic $ n $-manifold or orbifold is a subset of the manifold homeomorphic to $ B \times \mathbb{R}^+ $ for some cross-section $ B $.   Cusps result from the parabolic elements of $ \Gamma $ that fix a single point $ y $ of $ \partial \mathbb{H}^n $.  Specifically, since $ \text{Stab}_\Gamma(y) $ acts on a horosphere centered at $ y $, which has a flat geometry, the cross-section of the corresponding cusp is given by $ B = \mathbb{E}^{n-1} / \text{Stab}_\Gamma(y) $.  In this paper, we will consider only finite-volume hyperbolic manifolds, so $ B $ is compact.  Furthermore, if $ \mathbb{H}^n / \Gamma $ is orientable, then so is $ B $.  For more information on cusps of hyperbolic manifolds and the \textit{thick-thin decomposition}, we refer the reader to \cite[Ch. 12]{Ratcliffe}.
	\begin{defn}[Commensurability]
		Two subgroups $ \Gamma_1 $ and $ \Gamma_2 $ of a group $ \Gamma $ are \textit{commensurable} if $ \Gamma_1 \cap \Gamma_2 $ has finite index in both $ \Gamma_1 $ and $ \Gamma_2 $.  Two hyperbolic orbifolds $ \mathbb{H}^n / \Gamma_1 $ and $ \mathbb{H}^n / \Gamma_2 $ are \textit{commensurable} if $ \gamma \Gamma_1 \gamma^{-1} $ and $ \Gamma_2 $ are commensurable in $ \text{Isom}(\mathbb{H}^n) $ for some $ \gamma \in \text{Isom}(\mathbb{H}^n) $.
	\end{defn}
	Note that two orbifolds are commensurable if they share a common finite cover.
	\subsection{Arithmetic manifolds}
	Since we are working solely with cusped hyperbolic manifolds, all arithmetic hyperbolic manifolds in this paper are of \textit{simplest type}.  This allows us to use a simpler definition of arithmetic hyperbolic manifolds than the more involved general definition.  This is stated for example in Prop. 6.4.2 in \cite{Witte}, with the condition $ n \neq 3,7 $, although this condition is unnecessary.
	\begin{defn}[Arithmetic hyperbolic orbifold/arithmetic group]
		Let $ M $ be a finite-volume cusped hyperbolic $ n $-manifold with $ \pi_1(M) = \Gamma < \text{Isom}(\mathbb{H}^n) $.  Then $ M $ is arithmetic if there exists a quadratic form $ q $ of signature $ (n,1) $ such that $ A^{-1} \Gamma A < \text{Isom}(\mathbb{H}^n_q) $ is commensurable to $ SO_0(q, \mathbb{Z}) $, where $ A $ is the linear transformation that maps $ \mathbb{H}^n_q $ to $ \mathbb{H}^n $ isometrically.  We say $ \Gamma $ is arithmetic under the same condition, that is, when $ \Gamma $ is conjugate to a subgroup of $ \text{Isom}^+(\mathbb{H}^n_q) $ commensurable to $ SO_0(q, \mathbb{Z}) $.
		%A cusped hyperbolic $ n $-orbifold $ M $ is arithmetic if it is commensurable to $ \mathbb{H}^n_q / SO_0(q, \mathbb{Z}) $ for some quadratic form $ q $ of signature $ (n,1) $.  Its fundamental group $ \Gamma $ is arithmetic under the same condition, that is, when $ \Gamma $ is conjugate to a subgroup of $ \text{Isom}^+(\mathbb{H}^n) $ commensurable to $ SO_0(q, \mathbb{Z}) $.
	\end{defn}
	For the rest of the paper, we may refer to the arithmetic manifold $ \mathbb{H}^n_q / SO_0(q, \mathbb{Z}) $ as $ \mathbb{H}^n / SO_0(q, \mathbb{Z}) $ using this particular embedding, without ambiguity.
	
	To any cusped arithmetic hyperbolic $ n $-orbifold $ M $, we can associate the (non-unique) quadratic form $ q $ from the definition.  There are easily-checkable conditions on quadratic forms $ q_1 $ and $ q_2 $ that determine whether $ \Gamma_1 = SO_0(q_1, \mathbb{Z}) $ and $ \Gamma_2 = SO_0(q_2, \mathbb{Z}) $ are commensurable as subgroups of $ \text{Isom}(\mathbb{H}^n) $, identifying both $ \text{Isom}(\mathbb{H}^n_{q_1}) $ and $ \text{Isom}(\mathbb{H}^n_{q_2}) $ with $ \text{Isom}(\mathbb{H}^n) $, and are thus associated to the same manifolds.
	\begin{prop}[{\cite[Theorem 1]{MA}}]
		\label{prop:PEisComm}
		Let $ M_1 $ and $ M_2 $ be arithmetic hyperbolic orbifolds with associated quadratic forms $ q_1 $ and $ q_2 $ respectively.  Then $ M_1 $ and $ M_2 $ are commensurable if and only if $ q_1 $ and $ q_2 $ are projectively equivalent.
	\end{prop}
	One way to determine whether two quadratic forms $ q_1 $ and $ q_2 $ of signature $ (4,1) $ are projectively equivalent is to scale both so they have the same discriminant, and then compare Hasse-Witt invariants.  In particular, since such forms have odd rank, if $ q_i $ has discriminant $ -d_i $, then the form $ d_i q_i $ must have discriminant $ -1 $.  Thus, we can deal with rational equivalence rather than projective equivalence by associating to a commensurability class of arithmetic hyperbolic 4-manifolds a (non-unique) quadratic form $ q $ of discriminant $ -1 $.  Furthermore, by Lemma \ref{lemma:all_quat}, we can take $ q $ to be of quaternion type.  We summarize this discussion as follows.
	\begin{cor}
		Every commensurability class $ C $ of cusped arithmetic hyperbolic $ 4 $-orbifolds has an associated quadratic form $ q $ of quaternion type such that $ \mathbb{H}^4 / SO_0(q, \mathbb{Z}) $ lies in $ C $.
	\end{cor}
	\subsection{Systoles}
	\begin{defn}[Systole length]
		The systole length of a manifold $ M $ is the minimal length of a closed geodesic in $ M $.
	\end{defn}
	The arithmetic $ n $-manifolds we deal with in this paper have a minimum bound on the systole length.  The following proposition is an application of Corollary 1.3 from \cite{ERT}.
	\begin{prop}
		\label{prop:systole}
		There is a lower bound on the systole length of a cusped arithmetic hyperbolic 4-manifold.
	\end{prop}
	We will use this fact to show that certain finite-volume hyperbolic $ n $-manifolds are non-arithmetic.
	\section{Compact flat 3-manifolds}
	\label{sec:flat}
	Recall from \S \ref{subsec:hyp_mflds} that finite-volume cusped hyperbolic $ n $-manifolds $ M = \mathbb{H}^n / \Gamma $ have compact flat $ (n-1) $-manifolds $ B $ for the cross-sections of their cusps, and if $ M $ is orientable, then so is $ B $.  Considering only orientable manifolds, this means that hyperbolic 2-manifolds and 3-manifolds only have one type of cusp cross-section each, respectively $ S^1 $ and $ T^2 $.  However, there are six orientable compact flat 3-manifolds up to homeomorphism, which means there are six possible cusp cross-sections for an orientable finite-volume hyperbolic 4-manifold.  We give a brief description of each in the table below.
	\begin{center}
		\begin{longtable}{|p{2cm}|p{5cm}|p{3.5cm}|p{2cm}|}
			\caption{The 6 orientable compact flat 3-manifolds \cite{Pfaffle}.} \\
			\hline
			Manifold $ M $ & Fundamental domain & $ \pi_1(M) $ & $ \text{Hol}(\pi_1(M)) $ \\ \hline \hline
			3-torus
			& \begin{center} \includegraphics[scale=0.2]{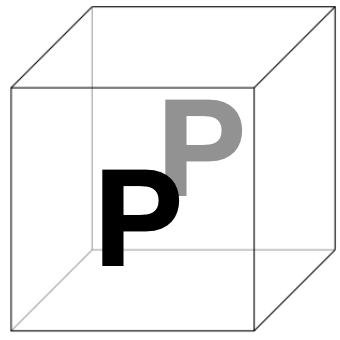} \end{center}
			& $ \mathbb{Z}^3 = \langle t_1, t_2, t_3 | \newline t_it_j = t_jt_i \rangle $
			& \textbf{1} \\ \hline
			$ \frac{1}{2} $-twist
			& \begin{center} \includegraphics[scale=0.2]{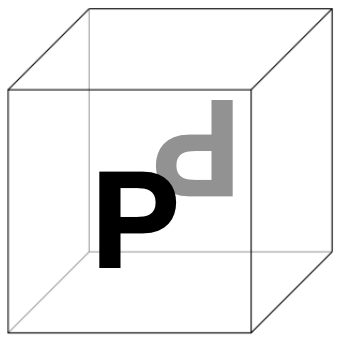} \end{center}
			& $ \langle \alpha, t_1, t_2, t_3 | \newline t_it_j = t_jt_i, \newline \alpha^2 = t_1, \newline \alpha t_2 \alpha^{-1} = t_2^{-1}, \newline \alpha t_3 \alpha^{-1} = t_3^{-1} \rangle $
			& $ \mathbb{Z} / 2\mathbb{Z} $ \\ \hline
			Hantzsche- \newline Wendt
			& \begin{center} \includegraphics[scale=0.25]{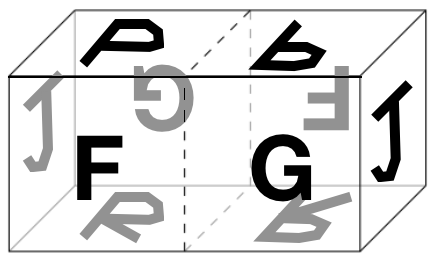} \end{center}
			& $ \langle x, y, z | \newline x y^2 x^{-1} y^2 = 1 \newline y x^2 y^{-1} x^2 = 1 \newline xyz = 1 \rangle $
			& $ \mathbb{Z}/2\mathbb{Z} \times \mathbb{Z}/2\mathbb{Z} $ \\ \hline
			$ \frac{1}{3} $-twist
			& \begin{center} \includegraphics[scale=0.15]{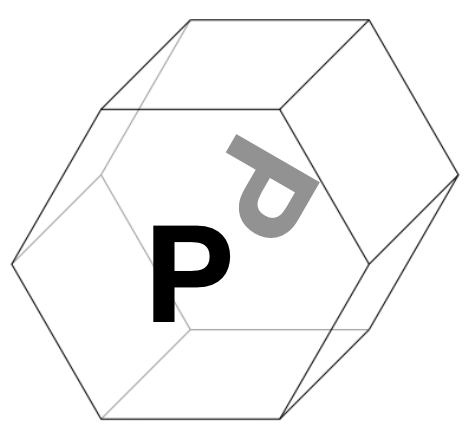} \end{center}
			& $ \langle \alpha, t_1, t_2, t_3 | \newline t_it_j = t_jt_i, \newline \alpha^3 = t_1, \newline \alpha t_2 \alpha^{-1} = t_3, \newline \alpha t_3 \alpha^{-1} = t_2^{-1} t_3^{-1} \rangle $
			& $ \mathbb{Z}/3\mathbb{Z} $ \\ \hline
			$ \frac{1}{4} $-twist
			& \begin{center} \includegraphics[scale=0.2]{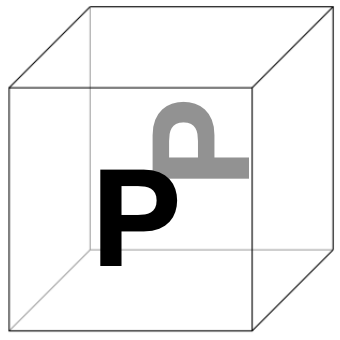} \end{center}
			& $ \langle \alpha, t_1, t_2, t_3 | \newline t_it_j = t_jt_i, \newline \alpha^4 = t_1, \newline \alpha t_2 \alpha^{-1} = t_3, \newline \alpha t_3 \alpha^{-1} = t_2^{-1} \rangle $
			& $ \mathbb{Z}/4\mathbb{Z} $ \\ \hline
			$ \frac{1}{6} $-twist
			& \begin{center} \includegraphics[scale=0.15]{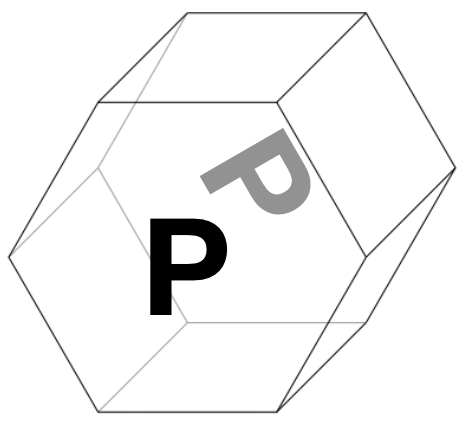} \end{center}
			& $ \langle \alpha, t_1, t_2, t_3 | \newline t_it_j = t_jt_i, \newline \alpha^6 = t_1, \newline \alpha t_2 \alpha^{-1} = t_3, \newline \alpha t_3 \alpha^{-1} = t_2^{-1} t_3 \rangle $
			& $ \mathbb{Z}/6\mathbb{Z} $ \\ \hline
		\end{longtable}
	\end{center}
	In the images depicting the fundamental domains, a face without a label is paired with its opposite face via translation, and labelled faces are paired in such a way that the labels align.  Note that all but the Hantzsche-Wendt manifold differ from the 3-torus by at most a twist on one of the face pairings.  All six flat manifolds are commensurable, and are in fact finitely covered by the 3-torus.
	
	Every isometry of Euclidean 3-space $ \mathbb{E}^3 $ is an affine transformation $ v \mapsto Av + w $ for some $ A \in SO(3) $.  For a group $ G < \text{Isom}(\mathbb{E}^3) $, the holonomy of $ G $ is given by 
	\begin{equation*}
		\text{Hol}(G) = \{ A \in SO(3) | (v \mapsto Av + w) \in G \text{~for~some~} w \in \mathbb{R}^3 \}.
	\end{equation*}
	$ \text{Hol}(G) $ is independent of the faithful representation of $ G $ into $ \text{Isom}(\mathbb{E}^3) $.
	\section{Classes with a given cusp}
	One goal of the next two sections is to prove Theorem \ref{theorem:infinite}.  In fact, we generalize Theorem \ref{theorem:infinite} to a full description of exactly when a commensurability class of cusped arithmetic hyperbolic 4-manifolds contains a manifold with a given cusp type.
	\begin{theorem}
		\label{theorem:full_pre}
		Let $ C $ be a commensurability class of cusped arithmetic hyperbolic 4-manifolds, with associated quadratic form $ q $, scaled so that the discriminant of $ q $ is $ -1 $.  Then:
		\begin{itemize}
			\item $ C $ must contain a manifold with a 3-torus cusp, a manifold with a $ \frac{1}{2} $-twist cusp, and a manifold with a Hantzsche-Wendt cusp.
			\item $ C $ contains a manifold with a $ \frac{1}{4} $-twist cusp if and only if $ \epsilon_p(q) = 1 $ for all \\ $ p \equiv 1 $ (mod $ 4 $).
			\item $ C $ contains a manifold with a $ \frac{1}{3} $-twist cusp if and only if $ \epsilon_p(q) = 1 $ for all \\ $ p \equiv 1 $ (mod $ 3 $).  $ C $ contains a manifold with a $ \frac{1}{6} $-twist cusp under the same condition.
		\end{itemize}
	\end{theorem}
	In this section, we prove the positive portion of the theorem, namely that $ C $ does indeed contain certain cusp types.
	\label{sec:cusp_yes}
	\begin{prop}
		\label{prop:cusp_yes}
		Let $ C $ be a commensurability class of arithmetic hyperbolic 4-manifolds, with associated quadratic form $ q $ of discriminant $ -1 $.  Then:
		\begin{itemize}
			\item $ C $ must contain a manifold with a 3-torus cusp, a manifold with a $ \frac{1}{2} $-twist cusp, and a manifold with a Hantzsche-Wendt cusp.
			\item If $ \epsilon_p(q) = 1 $ for all $ p \equiv 1 $ (mod $ 4 $), then $ C $ contains a manifold with a $ \frac{1}{4} $-twist cusp.
			\item If $ \epsilon_p(q) = 1 $ for all $ p \equiv 1 $ (mod $ 3 $), then $ C $ contains a manifold with a $ \frac{1}{3} $-twist cusp and a manifold with a $ \frac{1}{6} $-twist cusp.
		\end{itemize}
	\end{prop}
	Our primary tool for showing that a commensurability class must contain a given cusp type is the algorithm given by Long and Reid \cite{LR}.  Given a compact flat $ n $-manifold $ B $, this algorithm yields an arithmetic hyperbolic $ (n+1) $-orbifold with a cusp of type $ B $.  We can always find an $ (n+1) $-manifold with a cusp of type $ B $ covering this orbifold by \cite{McR}.
	
	Given a cusp type $ B $ of dimension $ n $, the algorithm works as follows.  Consider the holonomy group of $ \pi_1(B) $.  We can find a faithful representation of $ \text{Hol}(\pi_1(B)) $ into $ GL(n, \mathbb{Z}) $, which yields an embedding $ \text{Hol}(\pi_1(B)) \subset GL(n, \mathbb{Z}) $.  Further, we can choose a signature $ (n,0) $ quadratic form $ q_n $ that is invariant under $ \text{Hol}(\pi_1(B)) $ by considering an arbitrary signature $ (n,0) $ quadratic form $ r $, and taking the average of all the quadratic forms $ r \circ A $ over $ A \in \text{Hol}(\pi_1(B)) $, since $ \text{Hol}(\pi_1(B)) $ is finite.  Then, using linear algebra, the algorithm extends the representation into $ GL(n+2, \mathbb{Z}) $ in such a way that $ \text{Hol}(\pi_1(B)) $ leaves a quadratic form $ q^\prime $ rationally equivalent to $ q_n \oplus \langle 1, -1 \rangle $ invariant.  As a result, we see that some cover of $ \mathbb{H}^{n+1}/SO_0(q^\prime, \mathbb{Z}) $ must contain a cusp of type $ B $, and is commensurable to $ \mathbb{H}^{n+1}/SO_0(q_n \oplus \langle 1, -1 \rangle, \mathbb{Z}) $.
	
	By investigating properties of quadratic forms $ q_n $ invariant under $ \text{Hol}(\pi_1(B)) $, we will characterize the commensurability classes of arithmetic hyperbolic manifolds that can be output by this algorithm.  Since we're working with flat $ 3 $-manifolds and hyperbolic $ 4 $-manifolds, we apply the algorithm with $ n = 3 $.
	\begin{proof}[Proof of \ref{prop:cusp_yes}]
		Given the commensurability class $ C $, we can choose a quadratic form $ q = \langle x, y, xy, 1, -1 \rangle $ of quaternion type such that $ \mathbb{H}^4 / SO_0(q, \mathbb{Z}) \in C $ by Lemma \ref{lemma:all_quat}.  Note that $ q $ has discriminant $ -1 $.  We can compute the Hasse-Witt invariants $ \epsilon_p(q) $.
		
		First let $ B $ be the 3-torus, $ \frac{1}{2} $-twist, or Hantzsche-Wendt manifold.  These have holonomy groups of $ 1 $, $ \mathbb{Z} / 2\mathbb{Z} $, and $ \mathbb{Z} / 2\mathbb{Z} \times \mathbb{Z} / 2\mathbb{Z} $ respectively.  Each holonomy group has a representation into $ GL(3, \mathbb{R}) $ consisting solely of diagonal matrices with $ \pm 1 $ along the diagonal.  In particular, these representations fix any quadratic form $ \langle a, b, ab \rangle $ of rank 3.  Thus, we can apply the Long-Reid algorithm to find a representation of the corresponding Bieberbach group into $ SO_0(\langle a, b, ab, k, -k \rangle, \mathbb{Z}) $.  Set $ a = x $ and $ b = y $.  Then $ \langle a, b, ab, k, -k \rangle $ is rationally equivalent to $ \langle a, b, ab, 1, -1 \rangle = \langle x, y, xy, 1, -1 \rangle $.  This yields an orbifold commensurable to $ \mathbb{H}^4 / SO_0(q, \mathbb{Z}) $ that has the desired cusp type.  By \cite{McR}, there is also a manifold with the desired cusp type.
		
		Next, consider the $ \frac{1}{4} $-twist cusp.  This flat manifold has holonomy group $ \mathbb{Z} / 4\mathbb{Z} $, and has a representation $ \rho $ into $ SL(3, \mathbb{Z}) $ mapping its generator $ g_4 $ to
		\begin{equation*}
			\rho(g_4) = \left[ \begin{array}{ccc} 0 & 1 & 0 \\ -1 & 0 & 0 \\ 0 & 0 & 1 \end{array} \right].
		\end{equation*}
		This holonomy preserves any quadratic form $ q_3 = \langle a, a, b \rangle $, so the Long-Reid algorithm finds a representation of $ B $ into $ SO_0(\langle a, a, b, k, -k \rangle, \mathbb{Z}) $, which is commensurable to $ SO_0(\langle ab, ab, 1, 1, -1 \rangle, \mathbb{Z}) $.  The Hasse-Witt invariant at $ p $ of the form $ q^\prime = \langle ab, ab, 1, 1, -1 \rangle $ is equal to the Hilbert symbol $ (ab, ab)_p $.  Let $ ab = up^\alpha $, where $ u $ is an integer not divisible by $ p $.  By definition, for $ p > 2 $,
		\begin{equation*}
			(ab, ab)_p = (-1)^{\tau(p) \alpha \alpha}\left( \frac{u}{p} \right)^\alpha \left( \frac{u}{p} \right)^\alpha = (-1)^{\tau(p) \alpha},
		\end{equation*}
		where $ \tau(p) = \frac{p-1}{2} $.  Note $ \tau(p) $ is even if $ p \equiv 1 $ (mod 4) and odd if $ p \equiv 3 $ (mod 4).
		
		So if $ p \equiv 1 $ (mod 4), then we always have $ \epsilon_p(q^\prime) = (ab,ab)_p = 1 $.  But if $ p \equiv 3 $ (mod 4), then $ \epsilon_p(q^\prime) = -1 $ if and only if $ p $ divides $ ab $ an odd number of times.  Given the finite set of primes $ p_i > 2 $ such that $ \epsilon_p(q) = -1 $, as long as no such $ p_i \equiv 1 $ (mod 4), we can now ensure that there is a quadratic form $ q^{\prime \prime} = \langle \prod p_i, \prod p_i, 1, 1, -1 \rangle $ such that $ \epsilon_p(q^{\prime \prime}) = \epsilon_p(q) $.  Note that the identity $ \prod \epsilon_p(q) = 1 $ ensures that $ \epsilon_2(q^{\prime \prime}) = \epsilon_2(q) $ as well.  Thus $ q^{\prime \prime} $ and $ q $ both have the same Hasse-Witt invariants, as well as discriminant $ -1 $ and signature $ (4,1) $.  Hence $ q^{\prime \prime} $ is rationally equivalent to $ q $, and taking $ ab = \prod p_i $, we see $ \mathbb{H}^4 / SO_0(q^{\prime \prime}, \mathbb{Z}) $ must have a finite cover with a $ \frac{1}{4} $-twist cusp.  Thus we can construct a manifold in $ C $ with a $ \frac{1}{4} $-twist cusp.
		
		The arguments for the $ \frac{1}{3} $-twist and the $ \frac{1}{6} $-twist cusps are similar.  The holonomy groups $ \mathbb{Z}/3\mathbb{Z} $ and $ \mathbb{Z}/6\mathbb{Z} $ have representations $ \rho_3 $ and $ \rho_6 $ into $ SL(3, \mathbb{Z}) $ mapping the respective generators $ g_3 $ and $ g_6 $ as follows.
		\begin{equation*}
			\rho_3(g_3) = \left[ \begin{array}{ccc} -1 & -1 & 0 \\ 1 & 0 & 0 \\ 0 & 0 & 1 \end{array} \right], ~~~~~ \rho_6(g_6) = \left[ \begin{array}{ccc} 0 & -1 & 0 \\ 1 & 1 & 0 \\ 0 & 0 & 1 \end{array} \right].
		\end{equation*}
		Under this representation, both holonomy groups preserve quadratic forms of the form $ q^\prime(x) = 4ax_1^2 + 4ax_2^2 - 4ax_1x_2 + 3bx_3^2 $.  With some effort, we can show this form is projectively equivalent to $ q^{\prime \prime} = \langle ab, 3ab, 3, 1, -1 \rangle $.  We can compute that $ \epsilon_p(q^{\prime \prime}) = (ab, 3ab)_p (3, 3)_p $.  The second Hilbert symbol $ (3,3)_p $ is equal to $ -1 $ at $ p=2,3 $ and equal to 1 everywhere else.  To compute the first Hilbert symbol $ (ab, 3ab)_p $, we consider the case $ p = 3 $ separately from $ p \neq 2,3 $.  (We'll ignore $ p = 2 $ for now since the identity $ \prod \epsilon_p(q) = \prod \epsilon_p(q^{\prime \prime}) = 1 $ will ensure that $ \epsilon_2(q^{\prime \prime}) = \epsilon_2(q) $ if all other Hasse-Witt invariants are equal.)
		
		For $ p = 3 $, suppose $ ab = 3^\alpha u $ where $ u $ is not divisible by 3.  Then $ 3ab = 3^{\alpha+1}u $, so
		\begin{equation*}
			(ab, 3ab)_3 = (-1)^{\alpha(\alpha+1)\tau(3)} \left( \frac{u}{3} \right)^\alpha \left( \frac{u}{3} \right)^{\alpha+1} = \left( \frac{u}{3} \right).
		\end{equation*}
		Thus $ (ab, 3ab)_3 = 1 $ if $ u \equiv 1 $ (mod 3) and $ -1 $ if $ u \equiv 2 $ (mod 3).
		
		For $ p \neq 3 $, let $ ab = p^\alpha u $ where $ u $ is not divisible by $ p $, so that $ 3ab = p^\alpha (3u) $.  Then
		\begin{equation*}
			(ab, 3ab)_p = (-1)^{\alpha \alpha \tau(p)} \left( \frac{u}{p} \right)^\alpha \left( \frac{3u}{p} \right)^\alpha = \left((-1)^{\tau(p)} \left( \frac{3}{p} \right) \right)^\alpha.
		\end{equation*}
		Combining $ (-1)^{\tau(p)} $ with quadratic reciprocity, we can see that for $ p > 2 $, $ (ab, 3ab)_p = 1^\alpha $ if $ p \equiv 1 $ (mod 3), and $ (-1)^\alpha $ if $ p \equiv 2 $ (mod 3).  Consider the finite set of primes $ p_i > 2 $ such that $ \epsilon_p(q) = -1 $.  As long as no such $ p_i \equiv 1 $ (mod 3), we can take $ ab = \prod p_i $ over all $ p_i \equiv 2 $ (mod 3).  Additionally, we can multiply $ ab $ by 2 if necessary to set $ ab \equiv 1 $ or $ 2 $ (mod 3) to obtain the desired value of $ \epsilon_3(q^{\prime \prime}) $.  Now $ \epsilon_p(q) = \epsilon_p(q^{\prime \prime}) $ for all $ p > 2 $.  And as before, $ \epsilon_2(q) = \epsilon_2(q^{\prime \prime}) $, due to the identity $ \prod \epsilon_p(q) = \prod \epsilon_p(q^{\prime \prime}) = 1 $.  Now $ q^{\prime \prime} $ is rationally equivalent to $ q $, and from $ SO_0(q^{\prime \prime}, \mathbb{Z}) $ we can construct a manifold in $ C $ with $ \frac{1}{3} $-twist or $ \frac{1}{6} $-twist cusp, as desired.
	\end{proof}
	\begin{remark}
		In addition to the six orientable compact flat 3-manifolds, there are four non-orientable ones: two double-covered by the 3-torus and two double-covered by the $ \frac{1}{2} $-twist.  A thorough description of these manifolds can be found in \cite{ConRos}.  Notably, all of them have holonomies generated by orthogonal reflections.  In particular, this means each fundamental group has a holonomy representation into $ GL(3, \mathbb{R}) $ with image consisting of diagonal matrices with $ \pm 1 $ along the diagonal.  Thus, for the same reasons as the 3-torus, $ \frac{1}{2} $-twist, and Hantzsche-Wendt manifold, all four non-orientable compact flat 3-manifolds occur as a cusp cross-section in every commensurability class of arithmetic hyperbolic 4-manifolds.
	\end{remark}
	\section{Classes without a given cusp}
	\label{sec:cusp_no}
	The goal of this section is to prove the negative part of Theorem \ref{theorem:full_pre}; that is, to obstruct some cusp types from occurring in some commensurability classes of hyperbolic 4-manifolds.  This obstruction will yield infinitely many commensurability classes that avoid each of the $ \frac{1}{3} $-twist, $ \frac{1}{6} $-twist, and $ \frac{1}{4} $-twist.
	\begin{prop}
		\label{prop:cusp_no}
		Let $ C $ be a commensurability class of arithmetic hyperbolic 4-manifolds, with associated quadratic form $ q $ with discriminant $ -1 $.  Then:
		\begin{itemize}
			\item If $ \epsilon_p(q) \neq 1 $ for some $ p \equiv 1 $ (mod $ 4 $), then $ C $ does not contain a manifold with a $ \frac{1}{4} $-twist cusp.
			\item If $ \epsilon_p(q) \neq 1 $ for some $ p \equiv 1 $ (mod $ 3 $), then $ C $ does not contain a manifold with a $ \frac{1}{3} $-twist cusp, nor a manifold with a $ \frac{1}{6} $-twist cusp.
		\end{itemize}
	\end{prop}
	\begin{proof}
		By Lemma \ref{lemma:all_quat}, we can take $ q $ to be of quaternion form.  Thus, without loss of generality we can set $ q = \langle a, b, ab, 1, -1 \rangle $ for some positive integers $ a $ and $ b $.
		
		Let $ B $ be the cusp type that we want to obstruct, and let $ \Delta = \pi_1(B) $.  We will show that it suffices to obstruct the existence of an injective homomorphism $ \Delta \rightarrow SO_0(q, \mathbb{Q}) $.  
		
		For the sake of contradiction, suppose that $ C $ contains a manifold $ M $ with the cusp type in question. This yields an embedding $ \Delta \rightarrow \pi_1(M) = \Gamma $.  Because $ \Gamma $ is an arithmetic lattice in $ SO(4,1) $, we know that $ \Gamma $ lies in the $ \mathbb{Q} $-points of some quadratic form $ q^\prime $ \cite{Borel}.  Because $ M \in C $, $ q $ and $ q^\prime $ are projectively equivalent.  Thus by Prop. \ref{prop:PEisComm}, there exists a matrix $ F \in GL(5, \mathbb{Q}) $ such that $ F \Delta F^{-1} $ is commensurable with $ SO_0(q, \mathbb{Z}) $, and embeds into $ SO_0(q, \mathbb{Q}) $.  Note that $ \Delta $ acts on a horosphere centered at some point $ y $ in $ \partial \mathbb{H}^4 $.
		
		Since $ y $ is fixed by isometries that lie in $ SO_0(q, \mathbb{Z}) $, we can take $ y $ itself to lie in $ \mathbb{Q}^5 $.  Additionally, since $ SO_0(q, \mathbb{Q}) $ acts transitively on the rational points of $ \partial \mathbb{H}^4_q $,  we can choose $ y $ to be $ (0,0,0,1,1) $ without loss of generality.  Specifically, we can conjugate the image of $ \Delta $ by some matrix $ A^\prime \in SO_0(q, \mathbb{Q}) $ such that $ A^\prime y = y_0 = (0,0,0,1,1) $, to get a new rational representation of $ \Delta $ acting on a horosphere $ H $ centered at $ y_0 $.
		
		Let $ q_3 = \langle a, b, ab \rangle $ be the quadratic form such that $ q_3 \oplus \langle 1, -1 \rangle = q $.  Given any affine transformation $ \varphi \in \text{Isom}(\mathbb{E}^3) $, we can write the isometry as $ \varphi(v) = Av + w $, with $ A \in SO_0(q_3, \mathbb{R}) $.  Then we can map $ \varphi $ to an action $ \rho(\varphi) $ on $ H $ by taking $ \rho $ to be induced by an isometry from $ \mathbb{E}^3 $ to $ H $.  Imitating \cite{StackE}, we can write $ \rho $ as
		\begin{equation*}
			\rho(\varphi): v \mapsto \left[ \begin{array}{ccc} A & w & -w \\
				f(w)^t A & 1 + \frac{q_3(w)}{2} & -\frac{q_3(w)}{2} \\
				f(w)^t A & \frac{q_3(w)}{2} & 1 - \frac{q_3(w)}{2} \end{array} \right] v.
		\end{equation*}
		Here $ f $ is the linear function $ f(x) = (ax_1, bx_2, abx_3)^t $ such that $ f(x)^t x = q_3(x) $ for any $ x \in \mathbb{R}^3 $.  Since one can recover $ A $ from the top left and $ -w $ from the top right of $ \rho(\varphi) $, we see $ \rho $ must be injective.  One can check through manual calculation that $ \rho $ is a homomorphism, and that all elements in $ \rho(\text{Isom}(\mathbb{E}^3)) $ preserve both $ q $ and $ y_0 $, and thus act on $ H $.  All isometries of $ H $ must be of this form, so in particular, every element of $ \rho(\Delta) $ has this form.
		
		Depending on whether $ \Delta $ is the fundamental group of the $ \frac{1}{3} $-twist, $ \frac{1}{4} $-twist, or $ \frac{1}{6} $-twist cusp, it has holonomy group $ \mathbb{Z}/3\mathbb{Z} $, $ \mathbb{Z}/4\mathbb{Z} $, or $ \mathbb{Z}/6\mathbb{Z} $.  The holonomy is represented by the matrix $ A $ above, so in order to embed $ \Delta $ into $ SO_0(q, \mathbb{Q}) $, there must exist an isometry $ \varphi $ with $ A $ that is 3-torsion or 4-torsion.  Note that $ A $ has rational entries since it is a submatrix of $ \rho(\varphi) $, which has rational entries.  Thus, if we can obstruct 3-torsion or 4-torsion from $ SO_0(q_3, \mathbb{Q}) $, then we can obstruct the existence of an embedding $ \Delta \rightarrow SO_0(q, \mathbb{Q}) $.
		
		Now consider the quaternion algebra $ Q = \left( \frac{-a,-b}{\mathbb{Q}} \right) $.  The norm form of $ Q_0 $ is given by $ ax_1^2 + bx_2^2 + abx_3^2 = q_3(x) $, so by Theorem \ref{theorem:quad_to_quat}, $ SO(q_3, \mathbb{Q}) $ is isomorphic to $ Q^\ast / Z(Q^\ast) $.  Thus, if we obstruct torsion of some degree from appearing in $ Q^\ast / Z(Q^\ast) $, then we obstruct it from $ SO_0(q_3, \mathbb{Q}) < SO(q_3, \mathbb{Q}) $ as well.
		
		Now we apply Theorem \ref{theorem:12.5.6}.  For $ n = 3 $ and $ n = 4 $, clearly $ \xi_n + \xi_n^{-1} \in \mathbb{Q} $, and $ \xi_n \notin \mathbb{Q} $.  So there are no order $ n $ elements of $ Q^\ast / Z(Q^\ast) $ if and only if $ \mathbb{Q}(\xi_n) $ does not embed in $ Q $.  Furthermore, by Theorem \ref{theorem:7.3.3}, the field $ \mathbb{Q}(\xi_n) $ embeds in $ Q $ if and only if $ \mathbb{Q}(\xi_n) \otimes_\mathbb{Q} \mathbb{Q}_p $ is a field for each $ p \in \text{Ram}(Q) $.  The latter occurs exactly when $ p $ does not split in $ \mathbb{Q}(\xi_n) $.
		
		To check this condition, we must first determine when $ p \in \text{Ram}(Q) $.  If neither $ -a $ nor $ -b $ is divisible by $ p $ an odd number of times, then $ p $ does not ramify by Theorem \ref{theorem:2.6.6}(1).  Note that if both $ -a $ and $ -b $ are divisible by $ p $ an odd number of times, then $ ab $ is not.  Since $ a $, $ b $, and $ ab $ are interchangeable when constructing $ Q $, in this case we can pass to $ Q^\prime = \left( \frac{-a, -ab}{\mathbb{Q}} \right) $ to ensure that $ p $ divides only one of $ -a $ and $ -b $ an odd number of times.  Without loss of generality, say $ p $ divides $ -a $ but not $ -b $.  Then, by Theorem \ref{theorem:2.6.6}(2), $ p $ ramifies if and only if $ b $ is a square mod $ p $.
		
		We claim that $ p $ ramifies over $ Q $ exactly when the Hasse-Witt invariant $ \epsilon_p(q) = -1 $.  Using the definitions of the Hasse-Witt invariant and the Hilbert symbol, we can expand $ \epsilon_p(q) $.  Let $ a = p^\alpha j $ and $ b = p^\beta k $, with $ j $ and $ k $ relatively prime to $ p $.  Then $ ab = p^{\alpha + \beta} jk $, so:
		\begin{align*}
			\epsilon_p(q) & = \epsilon_p(\langle a, b, ab, 1, -1 \rangle) = (a,b)_p (ab, ab)_p \\
			& = \left[ (-1)^{\alpha \beta \tau(p)} \left( \frac{j}{p} \right)^\beta \left( \frac{k}{p} \right)^\alpha \right] \left[ (-1)^{(\alpha + \beta)(\alpha + \beta)\tau(p)} \left( \frac{jk}{p} \right)^{\alpha + \beta} \left( \frac{jk}{p} \right)^{\alpha + \beta} \right] \\
			& = (-1)^{\tau(p)(\alpha \beta + \alpha + \beta)} \left( \frac{j}{p} \right)^\beta \left( \frac{k}{p} \right)^\alpha
		\end{align*}
		If both $ \alpha $ and $ \beta $ are even, then $ \epsilon_p(q) = (-1)^0 ( \frac{j}{p} )^0 ( \frac{k}{p} )^0 = 1 $.  As shown above, $ p $ does not ramify over $ Q $ in this case.
		
		If both $ \alpha $ and $ \beta $ are odd, then we can choose to use $ Q^\prime = \left( \frac{-a, -ab}{\mathbb{Q}} \right) $ as before.  So unless both $ \alpha $ and $ \beta $ are even, without loss of generality we can assume $ \alpha $ is odd and $ \beta $ is even.  Then $ \epsilon_p(q) = (-1)^{\tau(p)} ( \frac{k}{p} ) $.  Note that $ ( \frac{-1}{p} ) $ is 1 when $ p \equiv 1 $ (mod 4) and $ -1 $ when $ p \equiv 3 $ (mod 4), so $ (-1)^{\tau(p)} = ( \frac{-1}{p} ) $.  Thus since $ b = p^\beta k $ with $ \beta $ even, we have $ \epsilon(p) = ( \frac{-1}{p} )( \frac{k}{p} ) = ( \frac{-k}{p} ) = (\frac{-b}{p}) $.  We already showed that $ p $ ramifies over $ Q $ exactly when $ -b $ is a nonsquare mod $ p $ in this case, which is equivalent to the condition $ \epsilon_p(q) = ( \frac{-b}{p} ) = -1 $.  Now in all cases, $ p $ ramifies over $ Q $ exactly when $ \epsilon_p(q) = -1 $.
		
		Next, we investigate when $ p $ splits in $ \mathbb{Q}(\xi_n) $.  When $ n = 3 $ or $ 6 $, we have $ \mathbb{Q}(\xi_n) = \mathbb{Q}(\sqrt{-3}) $, and if $ n = 4 $, we have $ \mathbb{Q}(\xi_n) = \mathbb{Q}(\sqrt{-1}) $.  It is well-known that $ p $ splits in $ \mathbb{Q}(\sqrt{a}) $ if and only if $ a $ is quadratic residue mod $ p $, so $ p $ splits in $ \mathbb{Q}(\sqrt{-3}) $ exactly when $ p \equiv 1 $ (mod 3), and in $ \mathbb{Q}(\sqrt{-1}) $ exactly when $ p \equiv 1 $ (mod 4).
		% Next we investigate when $ p $ splits in $ \mathbb{Q}(\xi_n) $.  We consider two cases: First, $ \mathbb{Q}(\xi_n) = \mathbb{Q}(\sqrt{-3}) $ when $ n = 3 $ or $ 6 $, and second, $ \mathbb{Q}(\xi_n) = \mathbb{Q}(\sqrt{-1}) $ when $ n = 4 $.  When $ p \equiv 1 $ (mod 3), then $ -3 $ is a quadratic residue mod $ p $.  In particular, $ P(x) = x^2 + 3 $ has a root mod $ p $, so $ P $ factors into two linear factors in $ \mathbb{Q}_p [x] $.  These factors are different since $ p \neq 2 $, and yield two different roots, corresponding to two primes of inertia degree 1 over $ p $.  Thus $ p $ splits in $ \mathbb{Q}(\sqrt{-3}) $ when $ p \equiv 1 $ (mod 3).  Similarly, when $ p \equiv 1 $ (mod 4), $ -1 $ is a quadratic residue mod $ p $, so $ P(x) = x^2 + 1 $ factors linearly mod $ p $.  So when $ p \equiv 1 $ (mod 4), then $ p $ splits in $ \mathbb{Q}(\sqrt{-1}) $.
		
		Now, suppose there is some prime $ p $ such that $ \epsilon_p(q) = -1 $ and $ p \equiv 1 $ (mod 4).  Then, $ p $ ramifies over $ Q $ and $ \mathbb{Q}(\xi_4) \otimes \mathbb{Q}_p $ is not a field.  Thus, as stated above, $ SO_0(q_3, \mathbb{Q}) \cong Q^\ast / Z(Q^\ast) $ has no 4-torsion.  As a result, the $ \frac{1}{4} $-twist group $ B $ cannot possibly embed into $ SO_0(q, \mathbb{Q}) $, so there is no $ \frac{1}{4} $-twist cusp in the associated commensurability class of hyperbolic 4-manifolds.  In fact, this same argument suffices to show there are no $ \frac{1}{4} $-twist cusps in the class of orbifolds, either.
		
		By similar logic, we can also see that if there is a prime $ p $ such that $ \epsilon_p(q) = -1 $ and $ p \equiv 1 $ (mod 3), then there is no 3-torsion in $ SO_0(q_3, \mathbb{Q}) \cong Q^\ast / Z(Q^\ast) $.  Thus the commensurability class of hyperbolic 4-manifolds (or orbifolds) associated to $ q $ must avoid $ \frac{1}{3} $-twist cusps and $ \frac{1}{6} $-twist cusps.
	\end{proof}
	Between Prop. \ref{prop:cusp_yes} and Prop. \ref{prop:cusp_no}, we've exhausted all possible commensurability classes in the case of each cusp type.  This suffices to prove Theorem \ref{theorem:full_pre}.  Theorem \ref{theorem:infinite} follows.
	\begin{example}
		Let $ q_6 = \langle 1, 1, 7, 7, -1 \rangle $.  The commensurability class of $ \mathbb{H}^4 / SO_0(q_6, \mathbb{Z}) $ avoids the $ \frac{1}{3} $-twist and $ \frac{1}{6} $-twist, since $ \epsilon_7(q_6) = -1 $, and $ 7 \equiv 1 $ (mod 3).
	\end{example}
	\begin{example}
		Let $ q_4 = \langle 1, 2, 5, 10, -1 \rangle $.  The commensurability class of $ \mathbb{H}^4 / SO_0(q_4, \mathbb{Z}) $ avoids the $ \frac{1}{4} $-twist since $ \epsilon_5(q_4) = -1 $, and $ 5 \equiv 1 $ (mod 4).
	\end{example}
	\section{Obstructions in higher dimensions}
	\label{sec:5dim}
	Using Theorem \ref{theorem:full_pre}, we can prove a version of Theorem \ref{theorem:infinite} one dimension higher.  Namely, some commensurability classes of hyperbolic 5-manifolds avoid some cusp types associated to flat 4-manifolds.  Our strategy will be to show that an arithmetic hyperbolic 5-manifold with cusp $ B \times S^1 $ must contain a 4-dimensional totally geodesic submanifold with cusp $ B $, and then manipulate Hasse-Witt invariants to show that no such submanifold can contain $ B $ as a cusp.
	\begin{prop}
		\label{lemma:must_B}
		Let $ B $ be either the $ \frac{1}{3} $-twist, $ \frac{1}{4} $-twist, or $ \frac{1}{6} $-twist.  Then any arithmetic hyperbolic 5-manifold $ M $ with $ B \times S^1 $ as a cusp cross-section contains an immersed finite-volume totally geodesic submanifold $ W $ with $ B $ as a cusp cross-section.
	\end{prop}
	\begin{proof}
		Let $ \Gamma $ be the fundamental group of $ M $.  As $ M $ is arithmetic, it is commensurable to some orbifold $ \mathbb{H}^5 / SO_0(q, \mathbb{Z}) $.  Let $ y $ be a light-like vector in $ \mathbb{H}^5_q $ that lies above the $ B \times S^1 $ cusp under the universal covering map of $ M $.
		
		The parabolic elements of $ \Gamma $ that fix $ y $ act on a horosphere $ E $ centered at $ y $ which is isomorphic to $ \mathbb{E}^4 $.  Without loss of generality, we can take $ E $ to be the horosphere passing through $ (0,0,0,0,0,1) $ by conjugating by an element of $ SO_0(q, \mathbb{Q}) $.  Note $ \text{Stab}_\Gamma(y) $ is isomorphic to $ \pi_1(B \times S^1) = \pi_1(B) \times \mathbb{Z} $, which acts on $ \mathbb{E}^3 \times \mathbb{E}^1 $.  We can choose a flat subspace $ P^\prime \subset E $ of dimension 3 such that $ H = \text{Stab}_\Gamma(y) \cap \text{Stab}_\Gamma(P^\prime) $ is isomorphic to $ \pi_1(B) $.  Let $ \gamma_1 $, $ \gamma_2 $, and $ \gamma_3 $ be three translations that generate the translation subgroup of $ H $.
		
		Unlike in $ \text{Isom}^+(\mathbb{H}^4) $, we can't assume that each $ \gamma_i $ lies in $ SO(q, \mathbb{Q}) $ for $ i = 1,2,3 $.  However, we can argue as follows.  The $ \gamma_i $ act by translation on $ E $, and so are parabolic translations.  One can check by applying $ \rho $ from Prop. \ref{prop:cusp_no} to any translation $ v \mapsto Iv + w $ that this means each $ \gamma_i $ must be unipotent as an element of $ SO_0(q, \mathbb{R}) $.  For each $ \gamma_i $, there is some positive integer $ k $ such that $ \gamma_i^k $ lies in $ SO_0(q, \mathbb{Z}) $, since $ \Gamma $ is commensurable to $ SO_0(q, \mathbb{Z}) $.  Hence, the field of coefficients of $ \gamma_i^k $, denoted $ F(\gamma_i^k) $, is $ \mathbb{Q} $.  This allows us to argue that $ F(\gamma_i) = \mathbb{Q} $, and so $ \gamma_i \in SO_0(q, \mathbb{Q}) $.  The justification of the previous sentence is somewhat technical, so we defer it to Lemma \ref{lemma:unipotent}.
		
		The three translations $ \gamma_i $ act on the three-dimensional subspace $ P^\prime \subset E $.  Since each $ \gamma_i \in SO_0(q, \mathbb{Q}) $, $ P^\prime $ must sit rationally in $ E \subset \mathbb{H}^5_q $.  To see this, pick any rational point in $ E $, say $ O = (0,0,0,0,1) $, and notice that each $ \gamma_i(O) \in \mathbb{Q}^5 $.
		
		The four points $ O $, $ \gamma_1(O) $, $ \gamma_2(O) $, and $ \gamma_3(O) $, together with $ y \in \mathbb{Q}^5 $, determine a four-dimensional hyperplane $ P $ which must also sit rationally in $ \mathbb{H}^5_q $.  Hence, after an appropriate change of basis over $ \mathbb{Q} $, the quadratic form $ q $ restricts to a rank 5 form $ f $ on the 5-dimensional subspace $ V \subset \mathbb{R}^6 $ containing $ P $.  Then since $ P $ consists of exactly the points in $ V $ satisfying $ f(x) = q(x) = -1 $ and $ x_6 > 0 $, $ P $ sits in $ V $ as $ \mathbb{H}^4_{f} $.  In particular, this means $ \text{Isom}^+(P) = SO_0(f, \mathbb{R}) $, so $ \text{Isom}^+(P) \cap SO_0(q, \mathbb{Z}) = SO_0(f, \mathbb{Z}) $.  Note this group is commensurable to $ \text{Isom}^+(P) \cap \Gamma $, as $ SO_0(q, \mathbb{Z}) $ is commensurable to $ \Gamma $.  Thus $ \text{Isom}^+(P) \cap \Gamma $ is arithmetic and its action on $ P $ has finite covolume.  Furthermore, $ \text{Isom}^+(P) \cap \text{Stab}_\Gamma(y) = \text{Isom}^+(P^\prime) \cap \text{Stab}_\Gamma(y) = H $, so $ W = P / (\text{Isom}^+(P) \cap \Gamma) $ has a cusp at $ y $ with cross-section $ B $.  Now, $ W $ is an immersed finite-volume totally geodesic submanifold of $ M $ with cusp $ B $.
	\end{proof}
	This completes the first half of the proof.  For the second, using Hasse-Witt invariants we prove that we should not find any totally geodesic 4-manifolds in our 5-manifold class with a cusp of type $ B $, yielding a contradiction.  The next step, then, is to find the Hasse-Witt invariants associated to such submanifolds.
	\begin{prop}
		\label{lemma:must_not_B}
		Let $ q $ be a quadratic form of signature $ (5,1) $, discriminant $ -1 $, and Hasse-Witt invariants $ \epsilon_p(q) $, and let $ M $ be a hyperbolic 5-manifold commensurable to $ \mathbb{H}^5 / SO_0(q, \mathbb{Z}) $.  Then any immersed finite-volume totally geodesic 4-dimensional submanifold $ W \subset M $ must be commensurable to $ \mathbb{H}^4 / SO_0(f, \mathbb{Z}) $, where $ f $ is a quadratic form of signature $ (4,1) $, discriminant $ -1 $, and Hasse-Witt invariants $ \epsilon_p(f) = \epsilon_p(q) $.
	\end{prop}
	\begin{proof}
		Since $ M $ is arithmetic, $ W $ is also arithmetic \cite[Theorem 3.2]{Meyer}.  Thus, we know $ W $ is commensurable to $ \mathbb{H}^4 / SO_0(f, \mathbb{Z}) $ for some quadratic form $ f $ of signature $ (4,1) $, which we can scale to ensure discriminant $ -1 $. All that remains to be shown is that $ \epsilon_p(f) = \epsilon_p(q) $ at all primes $ p $.
		
		Let $ f = \langle a, b, c, d, -abcd \rangle $ over a quadratic space with basis $ \{ v_1, \ldots, v_5 \} $.  Since $ W $ is an arithmetic manifold commensurable to $ \mathbb{H}^4 / SO_0(f, \mathbb{Z}) $, we know $ \pi_1(W) < SO_0(f, \mathbb{Q}) $ \cite{Borel}.  In particular, $ \pi_1(W) $ acts on $ \mathbb{H}^5 $ in such a way that it preserves $ f $ and a 4-dimensional hyperplane $ P $.  Taking a vector $ w $ transverse to $ P $, and adding it to the basis above, we have a basis $ \{ v_1, \ldots, v_5, w \} $ upon which we can define our quadratic form $ q $.  Though $ q $ may not be diagonal, we can use the Gram-Schmidt process to find a basis which makes $ q $ diagonal.  And since $ q $ restricted to $ \text{span}(\{ v_1, \ldots, v_5 \}) $ is already diagonal, the only basis element that is changed is $ w $.  Thus, since $ q $ has signature $ (5,1) $, it can be written as a diagonal form $ \langle a, b, c, d, -abcd, e \rangle $ for some positive $ e \in \mathbb{Z} $.  Since we started with the assumption that the discriminant of $ q $ is $ -1 $, we can conclude $ e = 1 $.
		
		It is now easy to show that the Hasse-Witt invariants of $ f = \langle a, b, c, d, -abcd \rangle $ are equal to the Hasse-Witt invariants of $ q = \langle a, b, c, d, -abcd, 1 \rangle $.  Since any Hilbert symbol $ (1,x)_p = 1 $,
		\begin{align*}
			\epsilon_p(q) & = (a,b)_p (a,c)_p (a,d)_p (a,-abcd)_p (b,c)_p (b,d)_p (b,-abcd)_p (c,d)_p (c,-abcd)_p (d,-abcd)_p \\ & ~~~~~~~~~~ (1,a)_p (1,b)_p (1,c)_p (1,d)_p (1,-abcd)_p \\
			& = (a,b)_p (a,c)_p (a,d)_p (a,-abcd)_p (b,c)_p (b,d)_p (b,-abcd)_p (c,d)_p (c,-abcd)_p (d,-abcd)_p \\
			& = \epsilon_p(f).
		\end{align*}
	\end{proof}
	\begin{theorem}
		\label{theorem:5D_full}
		Let $ B $ be either the $ \frac{1}{3} $-twist, $ \frac{1}{4} $-twist, or $ \frac{1}{6} $-twist.
		Then there exist commensurability classes of hyperbolic 5-manifolds that contain no manifolds with cusp cross-section given by $ B \times S^1 $.
	\end{theorem}
	\begin{proof}
		Consider any quadratic form $ q $ of signature $ (5,1) $ and discriminant $ -1 $.  We claim that if $ \epsilon_p(q) = -1 $ for any $ p \equiv 1 $ (mod 3), then the commensurability class $ C $ of $ \mathbb{H}^5 / SO_0(q, \mathbb{Z}) $ cannot contain $ B \times S^1 $ for $ B $ the $ \frac{1}{3} $-twist or the $ \frac{1}{6} $-twist, and if $ \epsilon_p(q) = -1 $ for any $ p \equiv 1 $ (mod 4), then this commensurability class cannot contain $ B \times S^1 $ for $ B $ the $ \frac{1}{4} $-twist.\label{key}
		
		By Prop. \ref{lemma:must_B}, any manifold $ M $ in $ C $ with a $ B \times S^1 $ cusp must contain an immersed totally geodesic submanifold $ W $ with a $ B $ cusp.  By Prop. \ref{lemma:must_not_B}, $ W $ must be commensurable to some $ \mathbb{H}^4 / SO_0(q^\prime, \mathbb{Z}) $ with $ \epsilon_p(q^\prime) = \epsilon_p(q) $ for all primes $ p $.  But by Theorem \ref{theorem:full_pre}, a manifold with these Hasse-Witt invariants cannot have a cusp with cross-section $ B $.  Thus, we've reached a contradiction, and such an $ M $ cannot exist in $ C $. 
	\end{proof}
	It is tempting to apply this argument repeatedly to find commensurability classes in higher-dimensional hyperbolic manifolds that avoid certain cusp types.  Unfortunately, this argument fails to work even in dimension 6, because Prop. \ref{lemma:must_not_B} fails to generalize.  Prop. \ref{lemma:must_not_B} relies on the fact that we can rescale a quadratic form of rank 5 to control the discriminant.  In rank 6, rescaling a quadratic form by $ k $ multiplies the discriminant by $ k^6 $, so the discriminant does not change in $ \mathbb{Q}^\ast / (\mathbb{Q}^\ast)^2 $.
	
	% In fact, one can show that any quadratic form of signature $ (7,1) $ is projectively equivalent to a form that looks like $ \langle a, b, ab, 1, 1, 1, 1, -1 \rangle $ by the same argument as in Lemma \ref{lemma:all_quat}.  Then, using the Long-Reid algorithm as in Prop. \ref{prop:cusp_yes}, we can show that for any 3-dimensional flat manifold $ B $, a cusp of type $ B \times (S^1)^3 $ occurs in every commensurability class of cusped orientable arithmetic hyperbolic 7-manifolds.  In particular, $ B $ and $ B \times (S^1)^3 $ have the same holonomy group, so we can use the same representation of the holonomy group into $ GL(n,3) $ as our holonomy representation, with some trivial blocks added along the main diagonal to make a $ 6 \times 6 $ matrix.  The same argument holds in higher than 7 dimensions as well.
	%More generally, we can even show that for a compact flat manifold $ B $ of any dimension $ n $, $ B \times (S^1)^k $ occurs as a cusp cross-section in every commensurability class of arithmetic hyperbolic $ (n+k) $-manifolds of simplest type for sufficiently high $ k $.  Thus, if we want to find cusp types of arbitrarily high dimension that are avoided by such commensurability classes, we'll have to find non-trivial high-dimensional flat manifolds.
	In fact, we can prove that repeatedly taking products of a compact flat manifold $ B $ with $ S^1 $ will eventually yield a manifold that occurs as a cusp cross-section in all arithmetic hyperbolic manifolds of the appropriate dimension.  Thus, if we want to find cusp types with obstructions in higher dimensions, we'll have to use non-trivial high-dimensional flat manifolds.
	\begin{theorem}
		\label{theorem:high_dim}
		Let $ B $ be a compact flat $ n $-manifold.  Then $ B \times (S^1)^k $ occurs as a cusp cross-section in every commensurability class $ C $ of cusped arithmetic hyperbolic $ (n+k+1) $-manifolds for sufficiently high $ k $.
	\end{theorem}
	\begin{proof}
		First, we prove the result for $ n+k+1 $ even.  When $ n+k+1 $ is even, any commensurability class $ C $ is associated with a quadratic form $ q $ of discriminant $ -1 $, since $ q $ has odd rank and we can scale $ q $ to control the discriminant.
		
		Note that $ B \times (S^1)^k $ has the same associated holonomy group as $ B $.  Since $ B $ is a flat manifold, the holonomy of its fundamental group $ \text{Hol}(\pi_1(B)) $ must be finite.  As such, $ \text{Hol}(\pi_1(B)) $ must be a subgroup of a symmetric group $ S_m $.  Let $ q_m $ denote the quadratic form $ \langle 1, \ldots, 1 \rangle $ of rank $ m $.  The natural representation $ \sigma $ of $ S_m $ into permutation matrices in $ GL(m,\mathbb{Z}) $ clearly preserves $ q_m $.  Restricting $ \sigma $ to $ \text{Hol}(\pi_1(B)) $, we have a representation of $ \text{Hol}(\pi_1(B)) $ that preserves $ q_m $, and must have entries in $ \mathbb{Z} $.  Let $ q_m^\prime = q_m \oplus \langle 1, -1 \rangle $.  We can use the Long-Reid algorithm \cite{LR} as in Prop. \ref{prop:cusp_yes} to construct an orbifold with cusp cross-section $ B \times (S^1)^k $ in the commensurability class of $ \mathbb{H}^{n+k+1}/SO_0(q^\prime \oplus q_m^\prime, \mathbb{Z}) $ for any positive definite quadratic form $ q^\prime $ of rank $ n+k-m \geq 0 $.
		
		Now, if $ m $ is even, let $ k = m-n+3 $ so that $ n+k+1 = m+4 $, and if $ m $ is odd, let $ k = m-n+4 $ so that $ n+k+1 = m+5 $.  This ensures $ n+k+1 $ is even.  Consider the class $ C $ of $ (n+k+1) $-manifolds with quadratic form $ q $ of discriminant $ -1 $.  We can show that $ q $ must be rationally equivalent to a quadratic form $ f = \langle a, b, c \rangle \oplus q_m^\prime $ (or $ f = \langle a, b, c, 1 \rangle  \oplus q_m^\prime $ if $ m $ is odd) by the same argument used to prove Lemma \ref{lemma:all_quat}, with $ q_m^\prime $ in the place of $ \langle 1, -1 \rangle $.  Then $ \mathbb{H}^{n+k+1} / SO_0(f, \mathbb{Z}) $ lies in $ C $, and is commensurable to a manifold with a cusp of type $ B \times (S^1)^k $.
		
		When $ n+k+1 $ is odd, we cannot control the discriminant of the quadratic form $ q $ associated to $ C $.  However, we can take a rank $ n+k $ subform $ q^\prime $ of $ q $ such that $ q = q^\prime \oplus \langle x \rangle $ for some positive integer $ x $.  Then we can scale $ q^\prime $ by $ y $ so that it has discriminant $ -1 $, and as in the paragraph above, $ q^\prime $ is rationally equivalent to $ f = \langle a, b, c \rangle \oplus q_m^\prime $ or $ f = \langle a, b, c, 1 \rangle \oplus q_m^\prime $.  But now $ yq = yq^\prime \oplus \langle xy \rangle $ is rationally equivalent to $ f \oplus \langle xy \rangle $, and we can conclude that $ \mathbb{H}^{n+k+1} / SO_0(f \oplus \langle xy \rangle) $ lies in $ C $ and is commensurable to a manifold with a cusp of type $ B \times (S^1)^k $, as before.
		%Now consider any commensurability class $ C $ of arithmetic hyperbolic $ (m+3) $-manifolds with associated quadratic form $ q $ of discriminant $ d < 0 $.  We can show that $ q $ must be projectively equivalent to a quadratic form $ f = \langle a, -ad \rangle \oplus q_m^\prime $ by the same argument used to prove Lemma \ref{lemma:all_quat}, with $ \langle 1, \ldots, 1, d \rangle $ in the place of $ \langle 1, -1 \rangle $.  Note that $ q_m^\prime = \langle 1, \ldots, 1, 1, -1 \rangle $ is projectively equivalent to $ \langle 1, \ldots, 1, -d, d \rangle $, so $ f $ is projectively equivalent to $ \langle a, -d, -ad, 1, \ldots, 1, d \rangle $.  Then $ \mathbb{H}^{m+3} / SO_0(f, \mathbb{Z}) $ lies in $ C $, and is commensurable to an orbifold with a cusp of type $ B \times (S^1)^k $, with $ k = m + 2 - n $.  By \cite{McR}, there is also a manifold in $ C $ with a cusp of type $ B \times (S^1)^k $.
	\end{proof}
	\begin{cor}
		Every commensurability class $ C $ of cusped arithmetic hyperbolic $ 8 $-manifolds contains a manifold with a cusp of type $ B \times (S^1)^3 $, where $ B $ is any compact flat $ 3 $-manifold.
	\end{cor}
	\begin{proof}
		By Theorem \ref{theorem:full_pre}, every $ B $ occurs in the commensurability class of \\ $ \mathbb{H}^4 / SO(\langle 1, 1, 1, 1, -1 \rangle, \mathbb{Z}) $.  The result follows from the third paragraph of the proof of Theorem \ref{theorem:high_dim}, using $ m = 3 $.
	\end{proof}
	\subsection{Fields of coefficients of unipotent matrices}
	In proving Prop. \ref{lemma:must_B}, we used the fact that for a unipotent matrix $ M $, the field of coefficients $ F(M) $, defined to be the number field obtained by adjoining the entries of $ M $ to $ \mathbb{Q} $, is unchanged under powers of $ M $.  We prove this result here.
	\begin{lemma}
		\label{lemma:unipotent}
		For any unipotent matrix $ M $ and any positive integer $ k $, $ F(M) = F(M^k) $.
	\end{lemma}
	\begin{proof}
		Because the entries of $ M^k $ are polynomial in the entries of $ M $, $ F(M^k) \subset F(M) $.  This holds for any $ M $, so in particular, $ F(M^{ak}) \subset F(M^k) $ for any nonnegative integer $ a $.  We will show that $ M $ can be written as a linear combination over $ \mathbb{Q} $ of matrices $ M^{ak} $, and thus that each entry in $ M $ is polynomial in entries of $ M^k $.  This will suffice to show $ F(M) \subset F(M^k) $. 
		
		By definition, a unipotent matrix $ M = I + T $, where $ T $ is a nilpotent matrix.  There is a positive integer $ l $ such that $ T^l = 0 $.  Now we can expand $ M^k = (I + T)^k $ using binomial coefficients.
		\begin{align*}
			M^k = \sum_{i=0}^k {k \choose i} T^i = \sum_{i=0}^{l-1} {k \choose i} T^i.
		\end{align*}
		Consider the vector space $ V $ over $ \mathbb{Q} $ consisting of the matrices spanned by all $ T^i $ for nonnegative integers $ i $.  $ V $ must have dimension at most $ l $, since only $ l $ of the $ T^i $ are nonzero.  We will show that if $ T^l = 0 $, then the $ l+1 $ matrices $ M^{ak} $ for $ a \in \{0,1, \ldots, l \} $ span $ V $.  Since $ M \in V $, this will show that $ M $ is a linear combination of these $ M^{ak} $.
		
		Choose some $ n \in \mathbb{Z}^+ $, and consider the following linear combination of matrices $ M^{ak} $.
		\begin{align*}
			\sum_{a=0}^n (-1)^{n+a} {n \choose a} M^{ak} & =
			\sum_{a=0}^n (-1)^{n+a} {n \choose a} \left[ \sum_{b=0}^{ak} {{ak} \choose b} T^b \right] \\
			& = \sum_{a=0}^n \sum_{b=0}^{ak} (-1)^{n+a} {n \choose a} {{ak} \choose b} T^b \\
			& = \sum_{b=0}^{nk} \sum_{a=\lceil \frac{b}{k} \rceil}^n (-1)^{n+a} {n \choose a} {{ak} \choose b} T^b \\
			& = \sum_{b=0}^{nk} (-1)^n \left[ \sum_{a=0}^n (-1)^a {n \choose a} {{ak} \choose b} \right] T^b
		\end{align*}
		Note that when we interchange the summations in line 3, we see that $ a $ is indexed from $ \lceil \frac{b}{k} \rceil $ to $ n $.  However, when $ ak < b $, $ {ak \choose b} = 0 $ anyway, so we can start $ a $ at 0 to get the same value.
		
		The coefficient of $ T^b $ in this sum is given by $ \sum_{a=0}^n (-1)^{n+a} {n \choose a} {{ak} \choose b} $.  Note that for fixed $ b $, $ {t \choose b} $ is a degree $ b $ polynomial in $ t $, defined over all nonnegative integers $ t $.  When $ b < n $, the coefficient of $ T^b $ is 0; we apply Lemma \ref{lemma:binomial}, proven below, with $ f(t) = {t \choose b} $ and $ y = k $.  Since the function $ g_f^{n,k}(x) $ is uniformly 0, it is 0 at $ x = 0 $ in particular.  Furthermore, when $ b = n $, $ {t \choose b} $ is a degree $ n $ polynomial $ a_nx^n + a_{n-1}x^{n-1} + \ldots + a_0 $.  The coefficient of $ T^b $ then must be $ a_n n! (-k)^n \neq 0 $, by Lemma \ref{lemma:binomial}.
		
		Now we can use induction on $ i $ to construct each $ T^i $ as a linear combination of $ M^{ak} $.  For the base case, consider $ i = l-1 $.  Choose $ n = l-1 $, and in the above summation, $ T^b $ has coefficient 0 when $ b < n = l-1 $, and $ T^b = 0 $ when $ b > n $ because $ b \geq l $.  Thus we've obtained a rational multiple of $ T^{l-1} $, which we can rescale to write $ T^{l-1} $ as a linear combination of $ M^{ak} $.
		
		For the induction step, assume $ T^j $ can be written as such a linear combination for all $ i < j \leq l-1 $.  Consider the linear combination above with $ n = i $.  Then, by Lemma \ref{lemma:binomial}, the coefficients of $ T^b $ are 0 for $ b < i $, nonzero for $ b = i $, and $ T^b = 0 $ for $ b \geq l $.  Since $ T^b $ can already be written as a linear combination for $ i < b \leq l-1 $ by the induction hypothesis, we can subtract out the appropriate linear combinations to leave only a multiple of $ T^i $.
		
		This suffices to show every $ T^i $ is a linear combination of $ M^{ak} $, and thus $ M = T^0 + T^1 $ is some linear combination of matrices $ M^{ak} $.  Since $ F(M^{ak}) \subset F(M^k) $ for all $ a $ and we have already proven $ F(M^k) \subset F(M) $, we conclude $ F(M) = F(M^k) $.
	\end{proof}
	Finally, we prove here the technical result that allowed us to conclude certain coefficients were zero or nonzero.
	\begin{lemma}
		\label{lemma:binomial}
		Let $ f: \mathbb{R} \rightarrow \mathbb{R} $ be a function, and fix $ y \in \mathbb{R} $ and $ n \in \mathbb{Z}^+ $.  Let
		\begin{equation*}
			g_f^{n,y}(x) = \sum_{a=0}^n (-1)^a f(x+ay) {n \choose a}.
		\end{equation*}
		If $ f $ is a polynomial of degree less than $ n $, then $ g_f^{n,y} = 0 $ uniformly.  Furthermore, if $ f(x) = x^n $, then $ g_f^{n,y} $ is the constant function $ n! (-y)^n $.
	\end{lemma}
	\begin{proof}
		First, we prove that $ g_f^{n,y} = 0 $ when $ f $ is a polynomial of degree less than $ n $ by induction on $ n $.  For the base case, consider $ n = 1 $.  In order for $ f $ to be a polynomial of degree less than 1, it must be a constant function $ f(x) = c $.  Then
		\begin{equation*}
			g_f^{1,y}(x) = \sum_{a=0}^1 (-1)^a f(x + ay) {1 \choose a} = f(x) - f(x+y) = c - c = 0.
		\end{equation*}
		Now assume the statement holds for $ n-1 $.  We can compute
		\begin{align*}
			g_f^{n,y}(x) & = \sum_{a=0}^n (-1)^a f(x+ay) {n \choose a} \\
			& = \sum_{a=1}^n (-1)^a f(x+ay) {{n-1} \choose {a-1}} + \sum_{a=0}^{n-1} (-1)^a f(x+ay) {{n-1} \choose a} \\
			& = - \sum_{a=0}^{n-1} (-1)^a f(x+y+ay) {{n-1} \choose a} + \sum_{a=0}^{n-1} (-1)^a f(x+ay) {{n-1} \choose a} \\
			& = -g_f^{n-1,y}(x+y) + g_f^{n-1,y}(x) \\
			& = \int_{x+y}^x \frac{\partial}{\partial t} \left[ g_f^{n-1,y}(t) \right] dt \\
			& = \int_{x+y}^x g_{f^\prime}^{n-1,y}(t) dt
		\end{align*}
		The second line above follows from the identity of binomial coefficients $ {n \choose a} = {{n-1} \choose {a-1}} + {{n-1} \choose a} $.  The final line follows from the fact that $ g_f^{n,y} $ is a particular linear combination of $ f(x+ay) $, with fixed coefficients depending on $ n $; concisely, $ g $ is linear in $ f $.  Since $ f $ is a polynomial of degree less than $ n $, its derivative $ f^\prime $ is a polynomial of degree less than $ n-1 $.  Thus, $ g_{f^\prime}^{n-1,y}(t) = 0 $ everywhere by induction, and therefore $ g_f^{n,y} = 0 $.
		
		Next, we prove that $ g_{f}^{n,y} = n! (-y)^n $ for $ f = x^n $ by induction on $ n $.  For the base case, consider $ n = 1 $.  Then
		\begin{equation*}
			g_f^{1,y}(x) = \sum_{a=0}^1 (-1)^a (x + ay) {1 \choose a} = (x) - (x+y) = -y.
		\end{equation*}
		Now assume the statement holds for $ n-1 $.  Let $ f(x) = x^n $ and $ h(x) = x^{n-1} $, so that $ f^\prime = nh $.  Then
		\begin{equation*}
			g_f^{n,y}(x) = \int_{x+y}^x g_{f^\prime}^{n-1,y}(t) dt = n \int_{x+y}^x g_{h}^{n-1,y}(t) dt \\ = n \int_{x+y}^x (n-1)!(-y)^{n-1} dt = n! (-y)^n.
		\end{equation*}
		We proved the first equality in the first part of this proof.  The rest follows from the fact that $ g_f^{n,y} $ is linear in $ f $, and the induction hypothesis.
	\end{proof}
	\section{Commensurability classes of non-arithmetic manifolds}
	\label{sec:non-arith}
	We can turn arithmetic commensurability classes that avoid certain cusp types into non-arithmetic ones by ``inbreeding'' the arithmetic manifolds with themselves, in a manner introduced by Agol \cite{Agol}.  We mimic the argument in \cite{Agol} to construct a manifold with arbitrarily short geodesic so that it must be non-arithmetic by Prop. \ref{prop:systole}.  Furthermore, this non-arithmetic group is constructed in such a way that it still lies in the $ \mathbb{Q} $-points of the original quadratic form, so we can conclude by the same argument as our proof of Prop. \ref{prop:cusp_no} that it avoids the same cusps.  Since this construction can be performed on any of the infinitely many classes that avoid the $ \frac{1}{3} $-twist, $ \frac{1}{4} $-twist, and $ \frac{1}{6} $-twist cusps, there are infinitely many non-arithmetic commensurability classes that avoid such cusps.
	\begin{proof}[Proof of Theorem \ref{theorem:non-arith}]
		Consider a quadratic form $ q $ such that the commensurability class of $ \mathbb{H}^4 / SO_0(q, \mathbb{Z}) $ does not contain any manifolds or orbifolds with a certain cusp $ B $.  Let $ M $ be any manifold in this commensurability class, and $ \Gamma $ its fundamental group.  By \cite[Theorem 4.2]{CR}, there exist infinitely many closed totally geodesic hyperbolic 3-manifolds immersed in $ M $.  These 3-manifolds lift to copies of $ \mathbb{H}^3 $ in $ \mathbb{H}^4 $; pick one such copy and call it $ P $.  Since the immersed 3-manifold is compact, $ H = \text{Isom}(P) \cap \Gamma $ acts cocompactly on $ P $.
		
		By Margulis' Commensurability Criterion for Arithmeticity \cite[Thm. 16.3.3]{Witte}, since $ \Gamma $ is arithmetic, its commensurator $ \text{Comm}(\Gamma) > PO(q, \mathbb{Q}) $.  Thus for any $ \epsilon > 0 $, we can choose $ \gamma \in \text{Comm}(\Gamma) $ such that $ \gamma(P) $ is disjoint from $ P $ and the distance $ d(P, \gamma(P)) < \frac{\epsilon}{2} $.  Since $ \gamma \in \text{Comm}(\Gamma) $, the stabilizer of $ \gamma(P) $, namely $ (\gamma H \gamma^{-1}) \cap \Gamma $, acts cocompactly on $ \gamma(P) $.  Then $ H_\gamma = \text{Isom}(\gamma(P)) \cap \Gamma $ must act cocompactly on $ \gamma(P) $, since $ (\gamma H \gamma^{-1}) \cap \Gamma < H_\gamma $.
		
		Let $ g $ be the geodesic segment orthogonal to both $ P $ and $ \gamma(P) $, intersecting $ P $ at $ p_1 $ and $ \gamma(P) $ at $ p_2 $.  Because $ H $ is discrete and residually finite, as a finitely generated linear group, we can choose a finite index subgroup $ H_1 < H $ so that $ d(p_1, h(p_1)) > 2 \text{arctanh}( \text{sech}( \frac{\epsilon}{4} )) $ for all non-identity $ h \in H_1 $.  Similarly, choose $ H_2 < H_\gamma $ so that $ d(p_2, h(p_2)) > 2 \text{arctanh}( \text{sech}( \frac{\epsilon}{4} )) $ for all non-identity $ h \in H_2 $.  Let $ \Sigma_1 = P / H_1 $ and $ \Sigma_2 = \gamma(P) / H_2 $.  Let $ E_i \in \mathbb{H}^4 $ be the Dirichlet domain of $ H_i $ centered at $ p_i $.
		
		Now, $ U = \Sigma_1 \cup_{p_1} g \cup_{p_2} \Sigma_2 $ is an embedded compact spine for $ E_1 \cap E_2 $, with one component of $ \mathbb{H}^4 - P $ retracting to $ \Sigma_1 $, the opposite component of $ \mathbb{H}^4 - \gamma(P) $ retracting to $ \Sigma_2 $, and the space in between $ P $ and $ \gamma(P) $ retracting to $ g $.
		
		Claim: $ G := \langle H_1, H_2 \rangle = H_1 \ast H_2 $, and $ G $ is geometrically finite.  We defer the proof to Lemma \ref{lemma:free_prod}.
		
		Then $ G $ is separable in $ \Gamma $ \cite{BHW}.  By Scott's separability criterion \cite{Scott}, for some finite index subgroup $ \Gamma_1 < \Gamma $, $ U $ embeds in $ \mathbb{H}^4 / \Gamma_1 $.  Thus, $ \Sigma_1 $ and $ \Sigma_2 $ embed in $ \mathbb{H}^4 / \Gamma_1 $.  Now let $ N = (\mathbb{H}^4 / \Gamma_1) - (\Sigma_1 \cup \Sigma_2) $, and $ D $ be the double of $ N $ along its boundary.  $ D $ is a hyperbolic manifold, since $ N $ is a hyperbolic manifold with totally geodesic boundary.  Note that the double of $ g $ is a closed geodesic of length $ \epsilon $, since $ g $ is perpendicular to $ \Sigma_1 $ and $ \Sigma_2 $.  Through choice of $ \epsilon $, we can construct $ D $ so that it has a geodesic of arbitrarily small length.  Thus by Prop. \ref{prop:systole}, we can construct $ D $ to be non-arithmetic.
		
		Next, we claim that $ \pi_1(D) < SO_0(q, \mathbb{Q}) $.  First, note that the universal cover of $ N $ is $ \mathbb{H}^4 $ with some half-spaces removed, with its group action given by $ \Gamma_1 $.  By construction, $ \Gamma_1 < \Gamma < SO_0(q, \mathbb{Q}) $.  Thus, we can find a fundamental domain $ S $ for $ N $ such that all the face pairings of $ S $ lie in $ SO_0(q, \mathbb{Q}) $.  We can construct a fundamental domain for $ D $ by taking two copies of $ S $, glued together at one of the boundary faces $ F $ that lifts to $ \Sigma_1 $, and pairing the remaining boundary faces by mapping each to its counterpart in the other copy of $ S $.  We will show that $ \pi_1(D) < SO_0(q, \mathbb{Q}) $ by showing that these face pairings, which generate $ \pi_1(D) $, each lie in $ SO_0(q, \mathbb{Q}) $.
		
		By construction, the face pairings $ \phi_i $ on the original copy of $ S $ must lie in $ SO_0(q, \mathbb{Q}) $.  The corresponding face pairings in the other copy of $ S $ are given by $ r_P \phi_i r_P $, where $ r_P $ is reflection across $ P $.  Recall that $ P $ was constructed as a hyperplane perpendicular to some $ v \in \mathbb{Q}^5 $, so the reflection $ r_P $ across $ P $ lies in $ SO_0(q, \mathbb{Q}) $.  Thus each $ r_P \phi_i r_P $ must also lie in $ SO_0(q, \mathbb{Q}) $.
		
		The remaining face pairings are the new ones formed from identifying boundary components of $ N $.  To pair a boundary component $ C $ with its corresponding mirror component, we can use the isometry $ r_P r_F $, where $ r_F $ is the reflection across the hyperplane $ F $ containing $ C $.  Note that $ F $ must be the image of $ \gamma(P) $ under some isometry $ \alpha \in \pi_1(N) $, so $ r_F = \alpha^{-1} r_{\gamma(P)} \alpha = \alpha^{-1} \gamma^{-1} r_P \gamma \alpha $.  Since we chose $ \gamma $ to lie in $ SO_0(q, \mathbb{Q}) $, and $ \alpha $ must be an element of $ \pi_1(S) $, $ r_F $ lies in $ SO_0(q, \mathbb{Q}) $ as well.  Now $ r_F r_P $ lies in $ SO_0(q, \mathbb{Q}) $, and thus every face pairing does as well.  Therefore, $ \pi_1(D) $ is generated by elements of $ SO_0(q, \mathbb{Q}) $, and so $ \pi_1(D) < SO_0(q, \mathbb{Q}) $.
		
		Now, if we choose the quadratic form $ q $ in such a way that the commensurability class of $ \mathbb{H}^4 / SO_0(q, \mathbb{Q}) $ avoids cusps with cross-section $ B $, then $ D $ cannot have cusps with cross-section $ B $, using the same argument as in the proof of Prop. \ref{prop:cusp_no}.  In this way, we construct using Theorem \ref{theorem:infinite} infinitely many commensurability classes of non-arithmetic manifolds that avoid the $ \frac{1}{3} $-twist, $ \frac{1}{4} $-twist, and $ \frac{1}{6} $-twist.
	\end{proof}
	The same proof can be applied to provide examples of commensurability classes of non-arithmetic hyperbolic 5-manifolds that avoid certain cusp types, with Theorem \ref{theorem:5D_full}.  Below we finish the proof by proving the claim we deferred.
	\begin{lemma}
		\label{lemma:free_prod}
		Let $ H_1 $ and $ H_2 $ be as in the above proof.  Then $ G = \langle H_1, H_2 \rangle $ is isomorphic to $ H_1 \ast H_2 $, and is geometrically finite.
	\end{lemma}
	\begin{proof}
		As in the proof of Theorem \ref{theorem:non-arith}, we let $ g $ be the geodesic segment connecting $ p_1 \in P $ with $ p_2 \in \gamma(P) $, meeting both planes perpendicularly.  Let $ L $ be the 3-plane that perpendicularly bisects $ g $, and consider the projections $ pr_1: \mathbb{H}^4 \rightarrow P $ and $ pr_2: \mathbb{H}^4 \rightarrow \gamma(P) $ that map each point in $ \mathbb{H}^4 $ to the closest point on the target 3-plane.  Using hyperbolic geometry (see Thm 3.5.10 in \cite{Ratcliffe}), we can see that $ pr_i(L) $ is a disk of radius $ \text{arctanh}( \text{sech}( \frac{\epsilon}{4} )) $ centered at $ p_i $.  We defined $ H_1 $ so that $ d(p_1, h(p_1)) > 2 \text{arctanh}( \text{sech}( \frac{\epsilon}{4} )) $ for all non-identity $ h \in H_1 $, so $ pr_1(L) $ must lie inside $ E_1 $, the Dirichlet domain of $ H_1 $ centered at $ p_1 $.  Thus, since $ H_1 < \text{Isom}(P) $, $ L $ must lie inside of $ E_1 $.  Similarly, $ L $ lies in $ E_2 $ as well.  Now $ L $ splits $ \mathbb{H}^4 $ into two parts, with $ \partial E_1 $ lying in the part with $ P $, and $ \partial E_2 $ lying in the part with $ \gamma(P) $.  Thus $ \partial E_1 \cap \partial E_2 = \emptyset $.  Since $ E_1 $ and $ E_2 $ are each geometrically finite, then, $ E_1 \cap E_2 $, the fundamental domain of $ G $, is geometrically finite, too.  Also, note that $ E_1 \cap E_2 = E_1 \# E_2 $, with the two sets glued along $ L $, so it's a fundamental domain of $ H_1 \ast H_2 $.  We can conclude that $ G $ is geometrically finite and $ G = H_1 \ast H_2 $.
	\end{proof}
	
	~ \\ ~ \\
	Department of Mathematics, \\
	Rice University, \\
	Houston, TX 77005, USA
	
	\textit{E-mail address}: \href{mailto:csell@rice.edu}{csell@rice.edu}
\end{document}